\documentclass{article}

\usepackage[a4paper]{geometry}
\usepackage[english,strings]{babel}
\usepackage[utf8]{inputenc}
\usepackage[T1]{fontenc}
\usepackage{amsmath}
\usepackage{amssymb}
\usepackage{amsthm}
\usepackage{amscd}
\usepackage{tikz}
\usetikzlibrary{cd}
\usepackage[all,cmtip]{xy}
\usepackage{hyperref}
\usepackage{enumitem}
\usepackage{mathrsfs}
\usepackage{mathtools}
\usepackage{authblk}

\theoremstyle{definition}
\newtheorem{theorem}			     {Theorem}		[section]
\newtheorem{proposition}  [theorem]	 {Proposition}

\newtheorem{remark} 	  [theorem]  {Remark}
\newtheorem{example}	  [theorem]  {Example}

\title{Galois theory and homology in quasi-abelian functor categories}

\author{Nadja Egner\footnote{The author’s research is funded by a FNRS doctoral grant of the Communauté française de Belgique.
}}

\affil{\small\textit{Institut de Recherche en Math\'ematique et Physique, Universit\'e catholique de Louvain, Chemin du Cyclotron 2, 1348 Louvain-la-Neuve, Belgium}}

\affil{nadja.egner@uclouvain.be}

\date{\today}

\begin{document}
	
\maketitle

\begin{abstract}
	Given a finite category $ \mathbb{T} $, we consider the functor category $ \mathscr{A}^\mathbb{T} $, where $ \mathscr{A} $ can in particular be any quasi-abelian category. Examples of quasi-abelian categories are given by any abelian category but also by non-exact additive categories as the categories of torsion(-free) abelian groups, topological abelian groups, locally compact abelian groups, Banach spaces and Fr\'{e}chet spaces. In this situation, the categories of various internal categorical structures in $ \mathscr{A} $, such as the categories of internal $ n $-fold groupoids, are equivalent to functor categories $ \mathscr{A}^\mathbb{T} $ for a suitable category $ \mathbb{T} $. For a replete full subcategory $ \mathbb{S} $ of $ \mathbb{T} $, we define $ \mathscr{F} $ to be the full subcategory of $ \mathscr{A}^\mathbb{T} $ whose objects are given by the functors $ F:\mathbb{T}\to\mathscr{A} $ with $ F(T)=0 $ for all $ T\notin\mathbb{S} $. We prove that $ \mathscr{F} $ is a torsion-free Birkhoff subcategory of $ \mathscr{A}^\mathbb{T} $. This allows us to study (higher) central extensions from categorical Galois theory in $ \mathscr{A}^\mathbb{T} $ with respect to $ \mathscr{F} $ and generalized Hopf formulae for homology.  
	
	\phantom{x}
	
	
\end{abstract}

\section*{Introduction}

Categorical Galois theory, as developed in \cite{janelidze:1990, janelidze:1991}, not only generalizes classical Galois theory but also establishes a link to the theory of central extensions of groups. These are surjective group homomorphisms $ f:A\to B $ whose kernel is contained in the center of $ A $. Given an admissible Galois structure $ \Gamma=(\mathscr{C},\mathscr{F},\mathsf{F},\mathsf{U},\mathscr{E},\mathscr{Z}) $, where 
\begin{equation*}
	\begin{tikzcd}
		\mathscr{C} \arrow[r, white, shift right=2, "\bot" black] \arrow[r, shift left=2, "\mathsf{F}"] &\mathscr{F} \arrow[l, shift left=2, "\mathsf{U}"]
	\end{tikzcd}
\end{equation*}
is an adjunction, and $ \mathscr{E} $ and $ \mathscr{Z} $ are classes of morphisms in $ \mathscr{C} $ and $ \mathscr{F} $, respectively, satisfying certain conditions, the notions of (trivial) coverings are introduced. In \cite{janelidze.kelly:1994}, the authors consider the situation where $ \mathscr{C} $ is an exact category, $ \mathscr{F} $ is a Birkhoff subcategory of $ \mathscr{C} $, and $ \mathscr{E} $ and $ \mathscr{Z} $ are the classes of regular epimorphisms in $ \mathscr{C} $ and $ \mathscr{F} $, respectively. In this case, (trivial) coverings are called (trivial) central extensions. In the case where $ \mathsf{F}:=\mathsf{Ab} $ is the abelianization functor from the category $ \mathscr{C}:=\mathrm{Grp} $ of groups to its subcategory $ \mathscr{F}:=\mathrm{Ab} $ of abelian groups, this general notion of central extension exactly recovers the notion of central extension from group theory. 

In \cite{gran:2001}, a characterization of the trivial and central extensions in the category $ \mathrm{Grpd}(\mathscr{C}) $ of internal groupoids in an exact Mal'tsev category $ \mathscr{C} $ with respect to the Birkhoff subcategory $ \mathrm{Discr}(\mathscr{C}) $ of discrete internal groupoids is given. Namely, an extension $ (f_0,f_1) $
\begin{equation*}
	\begin{tikzcd}
		C_1 \arrow[d, "f_1"'] \arrow[r, shift left=1, "d"] \arrow[r, shift right=1, "c"'] &C_0 \arrow[d, "f_0"]\\
		D_1 \arrow[r, shift left=1, "d"] \arrow[r, shift right=1, "c"'] &D_0
	\end{tikzcd}
\end{equation*}
between internal groupoids $ \mathbb{C} $ and $ \mathbb{D} $, meaning that $ f_0 $ and $ f_1 $ are regular epimorphisms in $ \mathscr{C} $, is central if and only if it is a discrete fibration, i.e., either of the above commutative squares is a pullback. If $ \mathscr{C} $ is protomodular, this is equivalent to the condition that the induced morphism between the kernels $ \mathrm{ker}(d) $ of the "domain" morphisms $ d $ of $ \mathbb{C} $ and $ \mathbb{D} $ is an isomorphism. In \cite{everaert.gran:2010}, $ \mathscr{C} $ is assumed to be semi-abelian in the sense of \cite{janelidze.marki.tholen:2002} and the (higher) central extensions in the category of $ n $-fold internal groupoids $ \mathrm{Grpd}^n(\mathscr{C}) $ with respect to $ \mathrm{Discr}(\mathscr{C}) $ are explicitly described. Moreover, generalized Hopf formulae are given. 

If $ \mathscr{C} $ is semi-abelian, the category $ \mathrm{Grpd}(\mathscr{C}) $ is equivalent to the category $ \mathrm{XMod}(\mathscr{C}) $ of internal crossed modules in $ \mathscr{C} $ \cite{janelidze:2003}. An internal crossed module is given by a morphism $ f:X\to B $ and an action $ \xi:B\flat X\to X $ of $ B $ on $ X $ satisfying the so-called equivariance and Peiffer conditions. Given an internal groupoid $ \mathbb{C} $ as above, the morphism part of the corresponding internal crossed module is given by $ c\circ \mathrm{ker}(d):\mathrm{Ker}(d)\to C_0 $. If $ \mathscr{C} $ is additive and has kernels, $ \mathrm{Grpd}(\mathscr{C}) $ is equivalent to the category $ \mathrm{RG}(\mathscr{C}) $ of internal reflexive graphs in $ \mathscr{C} $, and $ \mathrm{XMod}(\mathscr{C}) $ is equivalent to the arrow category $ \mathrm{Arr}(\mathscr{C}) $ of $ \mathscr{C} $. Consequently the categories $ \mathrm{Grpd}^n(\mathscr{C}) $ and $ \mathrm{XMod}^n(\mathscr{C}) $ are equivalent to the categories $ \mathrm{RG}^n(\mathscr{C}) $ and $ \mathrm{Arr}^n(\mathscr{C}) $, respectively. 

Restricting ourselves to the quasi-abelian setting \cite{yoneda:1960, schneiders:1999}, we characterize the central extensions in a general situation that includes the ones considered in $ \cite{gran:2001} $ and $ \cite{everaert.gran:2010} $. More specifically, we consider a finite category $ \mathbb{T} $ and a replete full subcategory $ \mathbb{S} $ of $ \mathbb{T} $. We show that the functor category $ \mathscr{A}^\mathbb{T} $ admits a torsion theory whose torsion-free subcategory is the full subcategory $ \mathscr{F} $ of $ \mathscr{A}^\mathbb{T} $ whose objects are given by the functors $ F:\mathbb{T}\to\mathscr{A} $ with $ F(T)=0 $ for all $ T\notin\mathbb{S} $. This implies that $ (\mathscr{A}^\mathbb{T},\mathscr{F}, \mathsf{F},\mathsf{U},\mathscr{E},\mathscr{Z}) $ yields an \textit{admissible} Galois structure \cite{janelidze:1990}, where $ \mathsf{F}:\mathscr{A}^\mathbb{T}\to\mathscr{F} $ is the reflection and $ \mathsf{U}:\mathscr{F}\to\mathscr{A}^\mathbb{T} $ is the inclusion of $ \mathscr{F} $, and $ \mathscr{E} $ and $ \mathscr{Z} $ are the classes of regular epimorphisms in $ \mathscr{A}^\mathbb{T} $ and $ \mathscr{F} $, respectively. Furthermore, $ \mathscr{F} $ is a Birkhoff subcategory of $ \mathscr{A}^\mathbb{T} $. We show that a regular epimorphism $ \alpha:F\to G $ in $ \mathscr{A}^\mathbb{T} $ is a central extension if and only if its component $ \alpha_T:F(T)\to G(T) $ is an isomorphism whenever $ T\notin\mathbb{S} $. We also study higher central extensions which are linked to generalized Hopf formulae for homology \cite{everaert.gran.vanderlinden:2008, duckerts-antoine:2017}. 

The article consists of three sections each of which comprises a part which recalls the theoretical background, and a part in which we study $ \mathscr{A}^\mathbb{T} $ and $ \mathscr{F} $ under different perspectives. Section~\ref{subsec:torsiontheories} recalls the notion of torsion theory in a pointed category. In Section~\ref{subsec:torsiontheoryTF}, we show that $ \mathscr{F} $ is a torsion-free subcategory of $ \mathscr{A}^\mathbb{T} $ whenever $ \mathscr{A} $ is quasi-abelian. In Section~\ref{subsec:categoricalGaloistheory}, we recall the notions of admissible Galois structure, and of trivial, normal and central extensions. In Section~\ref{subsec:categoricaltheoryofcentralextensions}, we recall how categorical Galois theory is used to develop a categorical theory of central extensions. In Section~\ref{subsec:centralextensionsinA^T}, we give an explicit description of the trivial extensions in $ \mathscr{A}^\mathbb{T} $ with respect to $ \mathscr{F} $ when $ \mathscr{A} $ is abelian, and of the central extensions when $ \mathscr{A} $ is quasi-abelian. We consider two concrete examples, where $ \mathbb{T} $ and $ \mathbb{S} $ are chosen in such a way that the pair $ (\mathscr{A}^\mathbb{T},\mathscr{F}) $ corresponds to $ (\mathrm{Arr}(\mathscr{A}),\mathscr{A}) $ and $ (\mathrm{Arr}^2(\mathscr{A}),2\textrm{-Arr}(\mathscr{A})) $, respectively. In Section~\ref{subsec:highercentralextensions}, we recall the generalization of the categorical theory of central extensions to higher extensions and then characterize those in $ \mathscr{A}^\mathbb{T} $ in Section~\ref{subsec:highercentralextensionsinA^T}.

\textsc{Acknowledgements.} The author warmly thanks George Janelidze for the fruitful discussions during her stay at the University of Cape Town in spring 2023.  

\section{Torsion theory $ (\mathscr{T},\mathscr{F}) $ in $ \mathscr{A}^\mathbb{T} $}\label{sec:torsiontheoryTF}

In this section, we will show that the functor category $ \mathscr{A}^\mathbb{T} $, where $ \mathbb{T} $ is a finite category and $ \mathscr{A} $ is a quasi-abelian category, allows for many torsion theories $ (\mathscr{T},\mathscr{F}) $. Namely, for any  replete full subcategory $ \mathbb{S} $ of $ \mathbb{T} $, the full subcategory $ \mathscr{F} $ of $ \mathscr{A}^\mathbb{T} $, whose objects are given by the functors $ F:\mathbb{T}\to\mathscr{A} $ such that $ F(T)=0 $ as soon as $ T\notin\mathbb{S} $, is torsion-free. We start with recalling some basic facts about torsion theories. 

\subsection{Torsion theories}\label{subsec:torsiontheories}

A \textbf{torsion theory} $ (\mathscr{T},\mathscr{F}) $ in a pointed category $ \mathscr{C} $ consists of two replete full subcategories $ \mathscr{T} $ and $ \mathscr{F} $ of $ \mathscr{C} $ such that the following conditions hold: 
\begin{enumerate}
	\item Any morphism $ f: T\to F $ in $ \mathscr{C} $ from an object $ T\in\mathscr{T} $ to an object $ F\in\mathscr{F} $ is $ 0 $.  
	\item For any object $ C\in\mathscr{C} $, there exists a short exact sequence 
	\begin{equation}\label{eq:sesTT}
		\begin{tikzcd}
			0\arrow[r] &T\arrow[r,"\varepsilon_C"] &C \arrow[r,"\eta_C"] &F\arrow[r] &0,
		\end{tikzcd}
	\end{equation}
	where $ T\in\mathscr{T} $ and $ F\in\mathscr{F} $. 
\end{enumerate}

It is easy to see that the short exact sequence \eqref{eq:sesTT} is unique up to unique isomorphism, i.e., given another such short exact sequence, there exist unique isomorphisms $ \varphi $ and $ \psi $ such that the following diagram commutes: 
\begin{equation*}
	\begin{tikzcd}
		0\arrow[r] &T\arrow[d, dotted, "\varphi"']\arrow[r,"\varepsilon_C"] &C \arrow[d, equal] \arrow[r,"\eta_C"] &F\arrow[d, dotted, "\psi"]\arrow[r] &0\\
		0\arrow[r] &T'\arrow[r,"\varepsilon'_C"'] &C \arrow[r,"\eta'_C"'] &F\arrow[r] &0
	\end{tikzcd}
\end{equation*}

Furthermore, one can show that $ \mathscr{T} $ is a normal mono-coreflective subcategory of $ \mathscr{C} $ whose coreflection $ \mathsf{T}:\mathscr{C}\to\mathscr{T} $ maps an object $ C $ to $ T $ as in \eqref{eq:sesTT}. The $ C $-component of the corresponding counit is given by $ \varepsilon_C $. Analogously, $ \mathscr{F} $ is a normal epi-reflective subcategory of $ \mathscr{C} $ whose reflection $ \mathsf{F}:\mathscr{C}\to\mathscr{F} $ maps an object $ C $ to $ F $ as in \eqref{eq:sesTT}. The $ C $-component of the corresponding unit is given by $ \eta_C $. Moreover, $ \mathsf{F} $ is semi-left exact \cite{cassidy.hebert.kelly:1985}, i.e., it preserves all pullbacks of the form
\begin{equation*}
	\begin{tikzcd}
		B\times_{\mathsf{F}(B)} X 
		\arrow[r, "p_2"]
		\arrow[d, "p_1"']
		& X
		\arrow[d, "\varphi"]\\
		B \arrow[r, "\eta_B"']
		&\mathsf{F}(B),
	\end{tikzcd}
\end{equation*}
where $ X $ lies in $ \mathscr{F} $ and $ \varphi $ is an arbitrary morphism. Here we omitted the inclusion functor $ \mathsf{U}:\mathscr{F}\to\mathscr{C} $. The following characterization of torsion-free subcategories can be found in \cite{janelidze.tholen:2007}. 

\begin{proposition}\label{prop:characterizationoftorsion-freesubcategories}
	Let $ \mathscr{C} $ be a pointed category with kernels and pullback stable normal epimorphisms, and $ \mathscr{F} $ be a replete full subcategory of $ \mathscr{C} $. Then the following conditions are equivalent: 
	\begin{enumerate}
		\item \label{it:torsion-freesubcategory} $ \mathscr{F} $ is a torsion-free subcategory of $ \mathscr{C} $. 
		\item \label{it:semi-left-exactreflection} $ \mathscr{F} $ is a normal epi-reflective subcategory of $ \mathscr{C} $ and the reflection $ \mathsf{F}:\mathscr{C}\to\mathscr{F} $ is semi-left exact. 
	\end{enumerate}
\end{proposition}

A torsion theory $ (\mathscr{T},\mathscr{F}) $ in a category $ \mathscr{C} $ is called \textbf{$ \mathscr{M} $-hereditary} for a class of monomorphisms $ \mathscr{M} $ in $ \mathscr{C} $, which is closed under composition with isomorphisms, if $ \mathscr{T} $ is closed under $ \mathscr{M} $-subobjects in $ \mathscr{C} $, meaning that, if $ m:A\to B $ is a morphism in $ \mathscr{M} $ and $ B $ is in $ \mathscr{T} $, then $ A $ is in $ \mathscr{T} $. We call a morphism a \textbf{protosplit monomorphism} if it is the kernel of some split epimorphism. We recall that a \textbf{homological category} is a category that is pointed, protomodular \cite{bourn:1991} and regular \cite{barr.grillet.vanosdol:1971}. A regular category is a finitely complete category that has coequalizers of kernels pairs and pullback stable regular epimorphisms. In the pointed context, protomodularity is equivalent to the validity to the Split Short Five Lemma. This means that, given a diagram 
\begin{equation*}
	\begin{tikzcd}
		A \arrow[d, "u"'] \arrow[r, "k"] &B \arrow[d, "v"'] \arrow[r, shift right=2, "f"'] &C \arrow[d, "w"] \arrow[l, shift right=2, "s"']\\
		A'\arrow[r, "k'"'] &B' \arrow[r, shift right=2, "f'"'] &C,' \arrow[l, shift right=2, "s'"']
	\end{tikzcd}
\end{equation*}
where $ fs=1_C $, $ f's'=1_{C'} $, $ k $ is the kernel of $ f $ and $ k' $ is the kernel of $ f' $, if $ u,w $ are isomorphisms, then $ v $ is an isomorphism, too. We can now recall \cite[Theorem 2.6]{everaert.gran:2014}. 

\begin{proposition}\label{prop:M-hereditarytorsiontheory}
	Let $ (\mathscr{T},\mathscr{F}) $ be a torsion theory in a homological category $ \mathscr{C} $. Then the following conditions are equivalent: 
	\begin{enumerate}
		\item The reflection $ \mathsf{F}:\mathscr{C}\to\mathscr{F} $ is protoadditive \cite{everaert.gran:2010}, i.e., it preserves split short exact sequences. 
		\item $ (\mathscr{T},\mathscr{F}) $ is $ \mathscr{M} $-hereditary, where $ \mathscr{M} $ is the class of protosplit monomorphisms in $ \mathscr{C} $. 
	\end{enumerate}
\end{proposition}

\subsection{Torsion theory $ (\mathscr{T},\mathscr{F}) $ in $ \mathscr{A}^\mathbb{T} $}\label{subsec:torsiontheoryTF}

In this section, we consider the functor category $ \mathscr{A}^\mathbb{T} $, where $ \mathbb{T} $ is a finite category and $ \mathscr{A} $ is a quasi-abelian category. Let us recall the definition of the latter notion. 

\begin{remark}[Quasi-abelian categories]
	Let $ \mathscr{A} $ be a preabelian category, i.e. an additive category that has kernels and cokernels. Then, for any morphism $ f:A\to B $, there exists a unique morphism $ \varphi:\mathrm{Cok}(\mathrm{ker}(f))\to\mathrm{Ker}(\mathrm{cok}(f)) $ such that the diagram
	\begin{equation*}
		\begin{tikzcd}
			A \arrow[dr, "\mathrm{cok}(\mathrm{ker}(f))"']\arrow[rrr, "f"]&&&B\\
			&\mathrm{Cok}(\mathrm{ker}(f)) \arrow[r, "\varphi"'] &\mathrm{Ker}(\mathrm{cok}(f)) \arrow[ur, "\mathrm{ker}(\mathrm{cok}(f))"']&
		\end{tikzcd}
	\end{equation*}
	commutes. The category $ \mathscr{A} $ is called \textbf{right semi-abelian} \cite{rump:2001} if $ \varphi $ is an epimorphism. Hence a preabelian category is right semi-abelian if and only if any arrow has an (epimorphism, normal monomorphism)-factorization.  Dually, $\mathscr{A}$ is called \textbf{left semi-abelian} if $ \varphi $ is a monomorphism, and \textbf{semi-abelian} \cite{gruson:1966, palamodov:1971, rump:2001} if it is both right and left semi-abelian. We note that this notion of semi-abelianness is different from the notion of a semi-abelian category in the sense of \cite{janelidze.marki.tholen:2002}. As it is shown in \cite[Proposition 1]{rump:2001}, $ \mathscr{A} $ is right semi-abelian if and only if the pushout of a normal monomorphism along any morphism is a monomorphism. $ \mathscr{A} $ is called \textbf{right quasi-abelian} (see \cite{rump:2001}, where, instead of quasi-abelian, the terminology \textit{almost abelian} is used) if normal monomorphisms are pushout stable, \textbf{left quasi-abelian} if normal epimorphisms are pullback stable, and \textbf{quasi-abelian} \cite{yoneda:1960, gruson:1966, raikov:1969, schneiders:1999, rump:2001} if it is both left and right quasi-abelian. Thus, $ \mathscr{A} $ is right (left) semi-abelian whenever $ \mathscr{A} $ is right (left) quasi-abelian. Moreover, a preabelian category is quasi-abelian if and only it is right (left) semi-abelian and left (right) quasi-abelian. As reported in \cite[5.3 Remarks]{rosicky.tholen:2007}, G. Janelidze observed that a category is quasi-abelian if and only if it is homological and co-homological, i.e., its opposite category is homological, too. In particular, any quasi-abelian category is regular. It is clear that any abelian category is an example of a quasi-abelian category. By the so-called Tierney's equation, a category is abelian if and only if it is (Barr-)exact \cite{barr.grillet.vanosdol:1971} and additive. We recall that an exact category is a regular category such that any equivalence relation is effective, i.e., any equivalence relation is isomorphic to a kernel pair of a morphism. In \cite[4 Corollary]{rump:2001}, it is shown that a category is quasi-abelian if and only if it is the torsion or torsion-free class in an abelian category. Hence the categories of torsion and torsion-free abelian groups are quasi-abelian categories that are not exact. Furthermore, a lot of examples arise in topological algebra and functional analysis. The categories of topological abelian groups, locally compact abelian groups, topological vector spaces, normed spaces, Banach spaces, locally convex spaces and Fr\'{e}chet spaces are quasi-abelian. Quasi-abelian categories provide a good framework for non-abelian homological algebra, K-theory and cohomological theory of sheaves \cite{schneiders:1999}.  
	
	For a long time, it was not clear if there were categories that were semi-abelian but not quasi-abelian. The so-called Raikov's conjecture that the axioms of a semi-abelian category already imply those of a quasi-abelian category (that Raikov called semi-abelian \cite{raikov:1969}) was proven false in 2005. In \cite{bonet.dierolf:2005}, the authors show that the category of bornological spaces is semi-abelian but not quasi-abelian. We refer the reader to \cite{kopylov.wegner:2012} for examples that are left semi-abelian (quasi-abelian) but not right semi-abelian (quasi-abelian) and vice versa. In \cite[2. Historical remark]{rump:2008}, the reader can found a historical remark on the history of research on semi-abelian and quasi-abelian categories. 
\end{remark} 

Given a replete full subcategory $ \mathbb{S} $ of the finite category $ \mathbb{T} $, we define $ \mathscr{F} $ to be the full subcategory of $ \mathscr{A}^\mathbb{T} $ whose objects are given by the functors $ F:\mathbb{T}\to\mathscr{A} $ such that $ F(T)=0 $ for all $ T\notin\mathbb{S} $. In particular, we can choose $ \mathbb{T} $ such that $ \mathscr{A}^\mathbb{T} $ is equivalent to $ \mathrm{Arr}^n(\mathscr{A}) $, and $ \mathbb{S} $ such that $ \mathscr{F} $ is equivalent to $ n\textrm{-Arr}(\mathscr{A}) $, which is the category whose objects are given by chains of $ n $ composable arrows such that each composite of two consecutive ones is $ 0 $. Namely, we can set $ \mathbb{T} $ to be $ \mathcal{P}(n)^\mathrm{op} $, where $ \mathcal{P}(n) $ denotes the partial order given by the power set of $ n $ viewed as a category. Here we consider the natural numbers by their standard (von Neumann) construction, and put $ 0:=\emptyset $ and $ n:=\{0,\ldots,n-1\} $ for $ n\geq1 $. As $ \mathbb{S} $, we may choose the full subcategory of $ \mathcal{P}(n)^\mathrm{op} $ whose objects are given by $ 0,\ldots,n $, see also Section~\ref{subsec:highercentralextensions} for more details. For $ n=3 $, an object in $ \mathrm{Arr}^3(\mathscr{A}) $ can be depicted as 
\begin{equation*}
	\begin{tikzcd}
		A_3 \arrow[rr] \arrow[dd] \arrow[dr] &&A_2 \arrow[dd] \arrow[dr] &\\
		&A_{\{1,2\}} \arrow[dd] \arrow[rr] &&A_{\{1\}} \arrow[dd]\\
		A_{\{0,2\}} \arrow[dr] \arrow[rr]  &&A_1 \arrow[dr] &\\
		&A_{\{2\}} \arrow[rr] &&A_0,
	\end{tikzcd}
\end{equation*}
and an object in $ 3\textrm{-Arr}(\mathscr{A}) $ as 
\begin{equation*}
	\begin{tikzcd}
		A_3 \arrow[r] &A_2 \arrow[r] &A_1 \arrow[r] &A_0.
	\end{tikzcd}
\end{equation*}
We note that $ \mathrm{Arr}^n(\mathscr{A}) $ and $ n\textrm{-Arr}(\mathscr{A}) $ are equivalent to $ \mathrm{Grpd}^n(\mathscr{A}) $ and $ n\textrm{-Grpd}(\mathscr{A}) $, respectively, whenever $ \mathscr{A} $ is additive and has kernels. 

Given an object $ T\in\mathbb{T} $, we denote by $ [T] $ the finite set of morphisms $ t:T'\to T $ with $ T'\notin\mathbb{S} $, and set $ d(t):=T' $ for such a morphism. For a functor $ F\in\mathscr{A}^\mathbb{T} $, we denote by 
\begin{equation*}
	[F(t)]_{t\in[T]}:\bigoplus_{t\in[T]}F(d(t))\to F(T)
\end{equation*}
the unique morphism such that $ [F(t)]_{t\in[T]}\iota_{t'}=F(t') $, where $ \iota_{t'}:F(d(t'))\to\bigoplus_{t\in[T]}F(d(t)) $ is the corresponding coproduct inclusion for $ t'\in[T] $. We now show that $ \mathscr{F} $ is the torsion-free subcategory of the torsion theory $ (\mathscr{T},\mathscr{F}) $ in $ \mathscr{A}^\mathbb{T} $, where $ \mathscr{T} $ is the full subcategory of $ \mathscr{A}^\mathbb{T} $ whose objects are given by the functors $ F\in\mathscr{A}^\mathbb{T} $ with $ [F(t)]_{t\in[T]} $ being an epimorphism for all $ T\in\mathbb{T} $.

\begin{proposition}\label{prop:torsiontheory}
	$ (\mathscr{T},\mathscr{F}) $ is a torsion theory in $ \mathscr{A}^\mathbb{T} $ whenever $ \mathscr{A} $ is a quasi-abelian categoy. 
\end{proposition}

\begin{proof}
	Firstly, let $ \alpha:F\to G $ be a morphism in $ \mathscr{A}^\mathbb{T} $ from an object $ F\in\mathscr{T} $ to an object $ G\in\mathscr{F} $. We show that $ \alpha=0 $. Let $ T $ be an object in $ \mathbb{T} $. We consider the commutative square 
	\begin{equation*}
		\begin{tikzcd}
			\bigoplus_{t\in[T]} F(d(t)) \arrow[rr, "{[}F(t){]}_{t\in{[}T{]}}"] \arrow[d, "\bigoplus_{t\in{[}T{]}}\alpha_{d(t)}"'] && F(T) 
			\arrow[d, "\alpha_T"]\\
			\bigoplus_{t\in[T]} G(d(t)) \arrow[rr, "{[}G(t){]}_{t\in{[}T{]}}"'] && G(T).
		\end{tikzcd}
	\end{equation*}
	By assumption, $ [F(t)]_{t\in[T]} $ is an epimorphism and $ \bigoplus_{t\in{[}T{]}}\alpha_{d(t)} $ is $ 0 $ since $ G(d(t))=0 $ for all $ t\in[T] $. Thus, $ \alpha_T=0 $. 
	
	Secondly, let $ F\in\mathscr{A}^\mathbb{T} $ be arbitrary. We define a short exact sequence 
	\begin{equation*}\tag*{($ \ast $)}
		\begin{tikzcd}
			0 \arrow[r] & \mathsf{T}(F) \arrow[r, "\varepsilon_F"] & F \arrow[r, "\eta_F"] & \mathsf{F}(F) \arrow[r] & 0
		\end{tikzcd}
	\end{equation*}
	in $ \mathscr{A}^\mathbb{T} $, where $ \mathsf{T}(F)\in \mathscr{T} $ and $ \mathsf{F}(F)\in\mathscr{F} $. For $ T\in\mathbb{T} $, we set $ \mathsf{F}(F)(T):=\mathrm{Cok}([F(t)]_{t\in[T]}) $ and $ (\eta_F)_T:=\mathrm{cok}([F(t)]_{t\in[T]}) $. For a morphism $ \tau:T\to T' $ in $ \mathbb{T} $, we define $ \mathsf{F}(\tau) $ to be the unique morphism such that the right-hand side of the diagram 
	\begin{equation*}
		\begin{tikzcd}
			{\bigoplus_{t\in[T]}F(d(t))} \arrow[rr, "{[F(t)]_{t\in[T]}}"] \arrow[d, "\varphi"'] && {F(T)} \arrow[d, "{F(\tau)}"] \arrow[rr, "{(\eta_F)_T}"] && {\mathsf{F}(F)(T)} \arrow[d, "{\mathsf{F}(F)(\tau)}", dotted] \\
			{\bigoplus_{t'\in[T']}F(d(t'))} \arrow[rr, "{[F(t')]_{t'\in[T']}}"'] && {F(T')} \arrow[rr, "{(\eta_F)_{T'}}"'] && {\mathsf{F}(F)(T')}
		\end{tikzcd}
	\end{equation*}
	commutes, where $ \varphi $ is the unique morphism such that $ \varphi\iota_t=\iota_{\tau t} $ for any $ t\in[T] $. It is immediate to see that $ \mathsf{F}(F) $ is a functor from $ \mathbb{T} $ to $ \mathscr{A} $ and that $ \eta_F $ is a natural transformation from $ F $ to $ \mathsf{F}(F) $. Furthermore, for $ T\notin\mathbb{S} $, $ \mathsf{F}(F)(T)=0 $ since in this case $ 1_T\in[T] $. Thus, $ \mathsf{F}(F)\in\mathscr{F} $.
	
	We set $ \mathsf{T}(F)(T):=\mathrm{Ker}((\eta_F)_T) $ for any $ T\in\mathbb{T} $ and $ (\epsilon_F)_T:=\mathrm{ker}((\eta_F)_T) $. For $ \tau:T\to T' $ in $ \mathbb{T} $, we define $ \mathsf{T}(F)(\tau) $ to be the unique morphism such that the front right-hand side of the diagram
	\begin{equation*}
		\begin{tikzcd}
			\bigoplus_{t\in[T]} F(d(t)) \arrow[dr, "e"'] \arrow[dd, "\varphi"'] \arrow[rr, "{[}F(t){]}_{t\in{[}T{]}}" near start] && F(T) \arrow[dd, "F(\tau)"] \arrow[r, "(\eta_F)_T"] &\mathsf{F}(F)(T) \arrow[dd, "\mathsf{F}(F)(\tau)"]\\
			& \mathsf{T}(F)(T) \arrow[dd, "\mathsf{T}(F)(\tau)" near start, dotted] \arrow[ur,  "(\epsilon_F)_T"']&&\\
			\bigoplus_{t'\in[T']} F(d(t')) \arrow[rd, "e'"'] \arrow[rr, "{[}F(t'){]_{t'\in{[}T'{]}}}"' near start] && F(T') \arrow[r,  "(\eta_F)_{T'}"'] &\mathsf{F}(F)(T'){,}\\
			&\mathsf{T}(F)(T') \arrow[ur, "(\epsilon_F)_{T'}"']
		\end{tikzcd}
	\end{equation*}
	commutes, i.e., $ (\varepsilon_F)_{T'}\mathsf{T}(F)(\tau)=F(\tau)(\varepsilon_F)_T $, where $ e $ and $ e' $ are the unique morphisms that make the upper and lower triangles commute, respectively. It is clear that $ \mathsf{T}(F) $ is a functor from $ \mathbb{T} $ to $ \mathscr{A} $ and that $ \varepsilon_F $ is a natural transformation from $ \mathsf{T}(F) $ to $ F $. Let $ T\in\mathbb{T} $ and $ t\in[T] $. For $ \tau=t $, the above diagram looks as follows:  
	\begin{equation*}
		\begin{tikzcd}
			\bigoplus_{t'\in[d(t)]} F(d(t')) \arrow[dr,  "{[}F(t'){]}_{t'\in{[}d(t){]}}"'] \arrow[dd, "\varphi"'] \arrow[rr, "{[}F(t'){]}_{t'\in{[}d(t){]}}" near start] && F(d(t)) \arrow[dd, "F(t)"] \arrow[r, ] &0 \arrow[dd]\\
			& F(d(t)) \arrow[dd, "\mathsf{T}(F)(t)" near start] \arrow[ur, equal]&&\\
			\bigoplus_{t''\in[T]} F(d(t'')) \arrow[rd, "e"'] \arrow[rr, "{[}F(t''){]_{t''\in{[}T{]}}}"' near start] && F(T) \arrow[r, "(\eta_F)_T"'] &\mathsf{F}(F)(T)\\
			&\mathsf{T}(F)(T') \arrow[ur, "(\epsilon_F)_T"']
		\end{tikzcd}
	\end{equation*} 
	We compute 
	\begin{equation*}
		(\varepsilon_F)_Te\iota_t=[F(t'')]_{t''\in[T]}\iota_t=F(t). 
	\end{equation*}
	Hence $ \mathsf{T}(F)(t)=e\iota_t $ and $ [\mathsf{T}(F)(t)]_{t\in[T]}=e $. Thus, $ \mathsf{T}(F)\in\mathscr{T} $ since $ e $ is, by assumption, an epimorphism. The exactness of the short sequence $ (\ast) $ is then clear. 
\end{proof} 

We note that the properties of a quasi-abelian category that we used in the proof of Proposition~\ref{prop:torsiontheory} were only the existence of finite coproducts, kernels and cokernels, and the fact that the induced morphism from the domain of a morphism to the kernel of its cokernel is an epimorphism.

As we have seen in Section~\ref{subsec:torsiontheories}, Proposition~\ref{prop:torsiontheory} shows that $ \mathscr{F} $ is a normal epi-reflective subcategory of $ \mathscr{A}^\mathbb{T} $ with semi-left exact reflection $ \mathsf{F}:\mathscr{A}^\mathbb{T}\to\mathscr{F} $ and unit $ \eta $. Since $ \mathsf{F} $ is a left adjoint between additive categories,  $ \mathsf{F} $ is protoadditive. Proposition~\ref{prop:M-hereditarytorsiontheory} implies that $ (\mathscr{T},\mathscr{F}) $ is $ \mathscr{M} $-hereditary, where $ \mathscr{M} $ is the class of protosplit monomorphisms in $ \mathscr{A}^\mathbb{T} $. We will see later that $ \mathscr{T} $ is not closed under regular subobjects. We recall from the proof of Proposition~\ref{prop:torsiontheory} that $ \mathsf{F}(F)(T)=\mathrm{Cok}([F(t)]_{t\in[T]}) $ and $ (\eta_F)_T=\mathrm{cok}([F(t)]_{t\in[T]}) $ for any $ F\in\mathscr{A}^\mathbb{T} $ and $ T\in\mathbb{T} $. 

\begin{remark}
	The full subcategory $ \mathscr{F} $ is a normal epi-reflective subcategory of $ \mathscr{A}^\mathbb{T} $ as soon as $ \mathbb{T} $ is a finite category and $ \mathscr{A} $ is a pointed category with finite coproducts and cokernels. Moreover, $ \mathsf{F}(F) $ is given by a pointwise left Kan extension as we explain now. Let us denote by $ \mathbb{T}^* $ the category with set of objects $ \mathrm{Ob}(\mathbb{T}^*):=\mathrm{Ob}(\mathbb{T})\dot{\cup}\{0\} $ and sets of morphisms 
	\begin{equation*}
		\mathrm{Hom}_{\mathbb{T}^*}(T,T'):=\begin{cases}
			\mathrm{Hom}(T,T')\dot{\cup}\{0\} &\textrm{if $ T,T'\in\mathbb{T} $,}\\
			\{0\} &\textrm{else}.
		\end{cases}
	\end{equation*}
	Let us denote by $ \mathbb{Q} $ the category with set of objects $ \mathrm{Ob}(\mathbb{Q}):=\mathrm{Ob}(\mathbb{S}^*)=\mathrm{Ob}(\mathbb{S})\dot{\cup}\{0\} $ and sets of morphisms 
	\begin{equation*}
		\mathrm{Hom}_\mathbb{Q}(S,S'):=\mathrm{Hom}_{\mathbb{S}^*}(S,S')/_\sim,
	\end{equation*}
	where we identify a morphism $ s:S\to S' $ with $ 0 $ if it factors in $ \mathbb{T} $ through an object which is not in $ \mathbb{S} $. This yields a functor $ \pi:\mathbb{T}^*\to\mathbb{Q} $ with 
	\begin{align*}
		\pi(T):=\begin{rcases}
			\begin{dcases}
				T &\textrm{if $ T\in\mathbb{S} $,}\\
				0 &\textrm{else}
			\end{dcases}
		\end{rcases},\quad
		&\pi(t):=\begin{rcases}
			\begin{dcases}
				[t] &\textrm{if $ t\in\mathbb{S} $,}\\
				0 &\textrm{else}
			\end{dcases}
		\end{rcases}.
	\end{align*}
	Let us fix a zero object $ 0_\mathscr{A} $ in $ \mathscr{A} $. Then we get the diagram 
	\begin{equation*}
		\begin{tikzcd}
			\mathscr{A}^\mathbb{T} \arrow[d, shift right=2, "\mathsf{L}"'] &&\mathscr{F} \arrow[ll, "\mathsf{U}"'] \arrow[d, shift right=2, "\overline{\mathsf{L}}"']\\
			\mathrm{Pointed}[\mathbb{T}^*,\mathscr{A}] \arrow[u, shift right=2, "\mathsf{R}"'] &&\mathrm{Pointed}[\mathbb{Q},\mathscr{A}], \arrow[ll, "\pi^*"] \arrow[u, shift right=2, "\overline{\mathsf{R}}"']
		\end{tikzcd}
	\end{equation*}
	where $ \mathrm{Pointed}[\mathbb{T}^*,\mathscr{A}] $ denotes the full subcategory of $ \mathscr{A}^{\mathbb{T}^*} $ whose objects are the pointed functors, and analogously for $ \mathrm{Pointed}[\mathbb{Q},\mathscr{A}] $. Furthermore, $ \pi ^*$ is the functor induced by precomposition with $ \pi $, and $ \mathsf{L},\mathsf{R},\overline{\mathsf{L}},\overline{\mathsf{R}} $ are the functors such that 
	\begin{align*}
		\mathsf{L}(F)(T):=\begin{rcases}
			\begin{dcases}
				F(T) &\textrm{if $ T\in\mathbb{T} $,}\\
				0_\mathscr{A} &\textrm{if $ T=0 $}
			\end{dcases}
		\end{rcases},\quad
		&\mathsf{R}(F):=F|_\mathbb{T},\\
		\overline{\mathsf{L}}(F)(S):=\begin{rcases}
			\begin{dcases}
				F(S) &\textrm{if $ S\in\mathbb{S} $,}\\
				0_\mathscr{A} &\textrm{if $ S=0 $}			
			\end{dcases}
		\end{rcases},\quad
		&\overline{\mathsf{R}}(F)(T):=\begin{rcases}
			\begin{dcases}
				F(T) &\textrm{if $ T\in\mathbb{S} $,}\\
				0_\mathscr{A} &\textrm{if $ T\notin\mathbb{S} $}
			\end{dcases}
		\end{rcases},
	\end{align*}
	and $ \mathsf{U}\overline{\mathsf{R}}\simeq\mathsf{R}\pi^* $ and $ \mathsf{LU}\simeq\pi^*\overline{\mathsf{L}} $. It is well-known \cite{maclane:1971} that, if $ \mathscr{A} $ is finitely cocomplete, $ \pi^* $ has a left adjoint $ \mathrm{Lan}_\pi:\mathrm{Pointed}[\mathbb{T}^*,\mathscr{A}]\to\mathrm{Pointed}(\mathbb{Q},\mathscr{A}) $ for which 
	\begin{equation*}
		(\mathrm{Lan}_\pi F)(S)=\underset{(T,f)\in\pi\downarrow S}{\mathrm{colim}}(F\circ \varphi_S)
	\end{equation*}
	for any $ F\in\mathrm{Pointed}[\mathbb{T}^*,\mathscr{A}] $ and $ S\in\mathbb{Q} $, where $ \pi\downarrow S $ is the category whose objects are the morphisms $ f:\pi(T)\to S $ in $ \mathbb{Q} $, where $ T\in\mathbb{T}^* $, and morphisms between $ (T,f) $ and $ (T',f') $ are the morphisms $ t:T\to T' $ in $ \mathbb{T}^* $ such that $ f'\pi(t)=f $, and $ \varphi_S:\pi\downarrow S\to \mathbb{T}^* $ is the expected forgetful functor. Since $ \mathscr{A}^\mathbb{T} $ and $ 	\mathrm{Pointed}[\mathbb{T}^*,\mathscr{A}] $, and $ \mathscr{F} $ and $ \mathrm{Pointed}[\mathbb{Q},\mathscr{A}] $ are equivalent, respectively, $ \mathsf{U} $ has $ \mathsf{R}\circ \mathrm{Lan}_\pi\circ \overline{\mathsf{L}} $ as left adjoint. It is easily shown that, for any $ S\in\mathbb{S} $, 
	\begin{equation*}
		\underset{(T,f)\in\pi\downarrow S}{\mathrm{colim}}(F\circ \varphi_S)=\mathrm{Cok}([F(t)]_{t\in[S]}). 
	\end{equation*}
\end{remark}

\section{Central extensions in $ \mathscr{A}^\mathbb{T} $ with respect to $ \mathscr{F} $}\label{sec:centralextensionsinA^T}

In this section, we will see that $ \mathscr{F} $ is an admissible subcategory of $ \mathscr{A}^\mathbb{T} $ from the point of view of categorical Galois theory. Furthermore, we will characterize the trivial and central extensions in $ \mathscr{A}^\mathbb{T} $ with respect to $ \mathscr{F} $ when $ \mathscr{A} $ is abelian and quasi-abelian, respectively. We start with recalling some basic facts from categorical Galois theory and the categorical theory of central extensions. 

\subsection{Categorical Galois theory}\label{subsec:categoricalGaloistheory}

Categorical Galois theory was introduced and developed in \cite{janelidze:1989, janelidze:1990, janelidze:1991}. A \textbf{Galois structure} $ \Gamma=(\mathscr{C},\mathscr{F},\mathsf{F},\mathsf{U},\mathscr{E},\mathscr{Z}) $ consists of an adjunction 
\begin{equation}\label{eq:Galoisstructure}
	\begin{tikzcd}
		{\mathscr{C}} 
		\arrow[r,""{name=1, anchor=center, inner sep=0}, "\mathsf{F}", shift left=2] 
		& {\mathscr{F}},
		\arrow[l, ""{name=0, anchor=center, inner sep=0}, "\mathsf{U}", shift left=2]
		\arrow["\dashv"{anchor=center, rotate=-90}, draw=none, from=1, to=0]
	\end{tikzcd}
\end{equation}
and classes $ \mathscr{E} $ and $ \mathscr{Z} $ of morphisms in $ \mathscr{C} $ and $ \mathscr{F} $, respectively, such that the following conditions hold: 
\begin{enumerate}
	\item $ \mathscr{C} $ and $ \mathscr{F} $ admit all pullbacks along morphisms in $ \mathscr{E} $ and $ \mathscr{Z} $, respectively. 
	\item $ \mathscr{E} $ and $ \mathscr{Z} $ contain all isomorphisms, are closed under composition and are pullback stable. 
	\item $ \mathsf{F}(\mathscr{E})\subseteq\mathscr{Z} $ and $ \mathsf{U}(\mathscr{Z})\subseteq\mathscr{\mathscr{E}} $. 
\end{enumerate}

A Galois structure $ \Gamma=(\mathscr{C},\mathscr{F},\mathsf{F},\mathsf{U},\mathscr{E},\mathscr{Z}) $ induces, for any object $ B\in\mathscr{C} $, the adjunction 
\begin{equation}
	\begin{tikzcd}\label{eq:inducedGaloisstructure}
		{\mathscr{E}(B)}
		\arrow[r, ""{name=1, anchor=center, inner sep=0}, "\mathsf{F}^B", shift left=2]
		& {\mathscr{Z}(\mathsf{F}(B)),}
		\arrow[l, ""{name=0, anchor=center, inner sep=0}, "\mathsf{U}^B", shift left=2]
		\arrow["\dashv"{anchor=center, rotate=-90}, draw=none, from=1, to=0]
	\end{tikzcd}
\end{equation}
where $ \mathscr{E}(B) $ is the full subcategory of the slice category $ \mathscr{C}\downarrow B $ whose objects are given by the morphisms in $ \mathscr{E} $ with codomain $ B $, and analogously for $ \mathscr{Z}(\mathsf{F}(B)) $. The left adjoint $ \mathsf{F}^B $ maps a morphism $ f:A\to B $ in $ \mathscr{E} $ to $ \mathsf{F}(f):\mathsf{F}(A)\to\mathsf{F}(B) $. For a morphism $ \varphi:X\to\mathsf{F}(B) $ in $ \mathscr{Z} $, we consider the pullback
\begin{equation}\label{eq:admissibilitypullback}
	\begin{tikzcd}
		B\times_{\mathsf{UF}(B)} \mathsf{U}(X) 
		\arrow[r, "p_2"]
		\arrow[d, "p_1"']
		& \mathsf{U}(X)
		\arrow[d, "\mathsf{U}(\varphi)"]\\
		B \arrow[r, "\eta_B"']
		&\mathsf{UF}(B)
	\end{tikzcd}
\end{equation}
of $ \mathsf{U}(\varphi) $ along the $ B $-component $ \eta_B $ of the unit of the adjunction \eqref{eq:Galoisstructure}. Then $ p_1 $ is the image of $ \varphi $ under the right adjoint $ \mathsf{U}^B $. By pasting the $ \mathsf{F} $-image of the above pullback with the corresponding naturality square of the counit $ \varepsilon $ of the adjunction \eqref{eq:Galoisstructure}, we see that the composite $ \varepsilon_X\mathsf{F}(p_2):\mathsf{F}(B\times_{\mathsf{UF}(B)} \mathsf{U}(X))\to X $ is a morphism from $ \mathsf{F}(p_1) $ to $ \varphi $ in $ \mathscr{Z}(\mathsf{F}(B)) $: 
\begin{equation*}
	\begin{tikzcd}
		\mathsf{F}(B\times_{\mathsf{UF}(B)} \mathsf{U}(X))
		\arrow[r, "\mathsf{F}(p_2)"]
		\arrow[d, "\mathsf{F}(p_1)"']
		& \mathsf{FU}(X)
		\arrow[d, "\mathsf{FU}(\varphi)"]
		\arrow[r, "\varepsilon_X"]
		&X
		\arrow[d, "\varphi"]\\
		\mathsf{F}(B) \arrow[r, "\mathsf{F}(\eta_B)"']
		&\mathsf{FUF}(B)
		\arrow[r,"\varepsilon_{\mathsf{F}(B)}"']
		&\mathsf{F}(B)
	\end{tikzcd}
\end{equation*} 
This is the $ \varphi $-component of the counit $ \varepsilon^B $ of the adjunction \eqref{eq:inducedGaloisstructure}. 

The Galois structure $ \Gamma=(\mathscr{C},\mathscr{F},\mathsf{F},\mathsf{U},\mathscr{E},\mathscr{Z}) $ is called \textbf{admissible} if $ \mathsf{U}^B $ is fully faithful for all $ B\in\mathscr{C} $. Equivalently, the counit $ \varepsilon^B $ is an isomorphism for all $ B\in\mathscr{C} $. When $ \mathsf{U} $ is fully faithful, then $ \Gamma $ is admissible if and only if $ \mathsf{F} $ preserves all pullbacks of the form \eqref{eq:admissibilitypullback}, where $ \varphi $ lies in $ \mathscr{Z} $.

If $ \Gamma=(\mathscr{C},\mathscr{F},\mathsf{F},\mathsf{U},\mathscr{E},\mathscr{Z}) $ is a Galois structure, where $ \mathscr{F} $ is a torsion-free subcategory of a pointed category $ \mathscr{C} $ with kernels and pullback stable normal epimorphisms with reflection $ \mathsf{F} $ and inclusion $ \mathsf{U} $, then Proposition~\ref{prop:characterizationoftorsion-freesubcategories} implies that $ \Gamma $ is admissible. 

Given an admissible Galois structure $ \Gamma=(\mathscr{C},\mathscr{F},\mathsf{F},\mathsf{U},\mathscr{E},\mathscr{Z}) $, the morphisms in $ \mathscr{E} $ are called \textbf{extensions}. Let $ f:A\to B $ be an extension. 
\begin{enumerate}
	\item $ f $ is called \textbf{trivial extension} if the naturality square 
	\begin{equation}\label{eq:naturalitysquare}
		\begin{tikzcd}
			A \arrow[d, "f"'] \arrow[r, "\eta_A"] &\mathsf{UF}(A) \arrow[d, "\mathsf{UF}(f)"]\\
			B \arrow[r, "\eta_B"'] &\mathsf{UF}(B)
		\end{tikzcd}
	\end{equation}
	is a pullback. Equivalently, $ f $ lies in the essential image of the functor $ \mathsf{U}^B:\mathscr{Z}(\mathsf{F}(B))\to\mathscr{E}(B) $. 
	\item $ f $ is called \textbf{monadic extension} if it is effective $ \mathscr{E} $-descent, i.e., the pullback functor $ f^*:\mathscr{E}(B)\to\mathscr{E}(A) $ is monadic. 
	\item $ f $ is called \textbf{normal extension} if it is a monadic extension and either of the projections $ p_1 $ and $ p_2 $ of its kernel pair
	\begin{equation*}
		\begin{tikzcd}
			A\times_B A \arrow[r, "p_2"] \arrow[d, "p_1"'] &A \arrow[d, "f"]\\
			A \arrow[r, "f"'] &B
		\end{tikzcd}
	\end{equation*}
	is a trivial extension. 
	\item $ f $ is called \textbf{central extension} if there exists a monadic extension $ p:E\to B $ such that $ p_1 $ in the pullback 
	\begin{equation*}
		\begin{tikzcd}
			E\times_B A \arrow[d, "p_1"'] \arrow[r, "p_2"] &A \arrow[d, "f"]\\
			E \arrow[r, "p"'] &B
		\end{tikzcd}
	\end{equation*}
	is a trivial extension. 
\end{enumerate}

The admissibility of $ \Gamma $ implies that trivial extensions are pullback stable. Furthermore, any trivial extension is normal and any normal extension is central.

\begin{remark}
	We want to shortly mention the fundamental theorem of categorical Galois theory due to G. Janelidze \cite{janelidze:1991}. It states that there is, for any admissible Galois structure $ \Gamma=(\mathscr{C},\mathscr{F},\mathsf{F},\mathsf{U},\mathscr{E},\mathscr{Z}) $ and any monadic extension $ p:E\to B $, an equivalence of categories $ \mathrm{Split}(E,p)\simeq \mathscr{F}^{\downarrow_\mathscr{Z}\mathrm{Gal}(E,p)} $ between the category of extensions $ f:A\to B $ whose pullback along $ p $ is a trivial extension, and the category of discrete fibrations between pregroupoids in $ \mathscr{F} $ with codomain the Galois pregroupoid $ \mathrm{Gal}(E,p) $ and with components in $ \mathscr{Z} $. When $ p:E\to B $ factors through every other monadic extension over $ B $, then we have that $ \mathrm{Split}(E,p) $ is exactly the category $ \mathrm{CExt}(B) $ of central extensions over $ B $. When $ \mathscr{C} $ is exact, has enough regular projective objects and $ \mathscr{E} $ is the class of regular epimorphisms in $ \mathscr{C} $, then we can choose for any $ B $ in $ \mathscr{C} $ such a $ p $. 
\end{remark} 

We conclude this section by studying a situation where the central extensions can explicitly be described. 

\begin{proposition}\label{prop:centralextensionsprotoadditive}
	Let $ \Gamma=(\mathscr{C},\mathscr{F},\mathsf{F},\mathsf{U},\mathscr{E},\mathscr{Z}) $ be an admissible Galois structure, where $ \mathscr{C} $ is a pointed protomodular category, and $ \mathscr{F} $ is a replete full reflective subcategory of $ \mathscr{C} $ with protoadditive reflection $ \mathsf{F} $ and inclusion $ \mathsf{U} $. Furthermore, we assume that any morphism in $ \mathscr{E} $ is a monadic extension. For an extension $ f:A\to B $, the following conditions are equivalent:
	\begin{enumerate}
		\item \label{it:normalextension} $ f $ is a normal extension. 
		\item \label{it:centralextension} $ f $ is a central extension. 
		\item \label{it:kerneltorsionfree} $ \mathrm{Ker}(f) $ lies in $ \mathscr{F} $. 
	\end{enumerate}
\end{proposition}
\begin{proof}
	For the convenience of the reader, we repeat the argument used in \cite[2.7. Proposition]{everaert.gran:2010}. By definition, \eqref{it:normalextension} implies \eqref{it:centralextension}. Let us assume that $ f:A\to B $ is a central extension, i.e., there exists a monadic extension $ p:E\to B $ such the pullback $ p_1 $ of $ f $ along $ p $ is a trivial extension. This means that the naturality square 
	\begin{equation*}
		\begin{tikzcd}
			E\times_B A \arrow[d, "p_1"'] \arrow[r, "\eta_{E\times_B A}"] &\mathsf{F}(E\times_B A) \arrow[d, "\mathsf{F}(p_1)"]\\
			E \arrow[r, "\eta_E"'] &\mathsf{F}(E)
		\end{tikzcd}
	\end{equation*}
	is a pullback. We have that the kernels of $ f $ and $ p_1 $ and $ \mathsf{F}(p_1) $ are isomorphic. Since $ \mathsf{F} $ is assumed to be protoadditive and $ p_1 $ is a split epimorphism, we have that $ \mathrm{Ker}(\mathsf{F}(p_1))\cong\mathsf{F}(\mathrm{Ker(p_1)}) $. Thus, $ \mathrm{Ker}(f) $ lies in $ \mathscr{F} $. It remains to show that \eqref{it:kerneltorsionfree} implies \eqref{it:normalextension}. Let us assume that the kernel of $ f $ lies in $ \mathscr{F} $. By assumption, $ f $ is a monadic extension. We show that $ p_1 $ in the pullback
	\begin{equation*}
		\begin{tikzcd}
			A\times_B A \arrow[d, "p_1"']\arrow[r, "p_2"] &A \arrow[d, "f"]\\
			A \arrow[r, "f"'] &B
		\end{tikzcd}
	\end{equation*}
	is a trivial extension, i.e., the naturality square 
	\begin{equation*}
		\begin{tikzcd}
			A\times_B A \arrow[d, "p_1"'] \arrow[r, "\eta_{A\times_B A}"] &\mathsf{F}(A\times_B A) \arrow[d, "\mathsf{F}(p_1)"]\\
			A \arrow[r, "\eta_A"'] &\mathsf{F}(A)
		\end{tikzcd}
	\end{equation*} 
	is a pullback. As before, we have that the kernels $ \mathrm{Ker}(f) $ of $ f $ and $ \mathrm{Ker}(p_1) $ of $ p_1 $ are isomorphic, and that $ \mathrm{Ker}(\mathsf{F}(p_1))\cong \mathsf{F}(\mathrm{Ker}(p_1)) $. Since $ \mathrm{Ker}(f) $ lies in $ \mathscr{F} $, it follows that $ \mathrm{Ker}(\mathsf{F}(p_1))\cong\mathrm{Ker}(p_1) $. By protomodularity, the naturality square is a pullback. 
\end{proof}

\subsection{Categorical theory of central extensions}\label{subsec:categoricaltheoryofcentralextensions}

Given a regular category $ \mathscr{C} $, a replete full reflective subcategory $ \mathscr{F} $ of $ \mathscr{C} $ is closed under subobjects in $ \mathscr{C} $, i.e., given a monomorphism $ f:A\to B $ in $ \mathscr{C} $, where $ B $ lies in $ \mathscr{F} $, then $ A $ lies in $ \mathscr{F} $, if and only if each component of the unit of the reflection is a regular epimorphism.  In this case, $ \mathscr{F} $ is a regular category, and a morphism in $ \mathscr{F} $ is a regular epimorphism (monomorphism) in $ \mathscr{F} $ if and only if it is a regular epimorphism (monomorphism) in $ \mathscr{C} $. However, when $ \mathscr{C} $ is exact, i.e., it has additionally effective equivalence relations, $ \mathscr{F} $ is not necessarily exact, as the example of  the subcategory of torsion-free abelian groups of the category of abelian groups shows \cite{borceux:1994}. If $ \mathscr{F} $ is a \textbf{Birkhoff subcategory} of $ \mathscr{C} $, i.e., if it is in addition closed under regular quotients in $ \mathscr{C} $, meaning that, if $ f:A\to B $ is a regular epimorphism in $ \mathscr{F} $, where $ A $ is in $ \mathscr{F} $, then $ B $ is in $ \mathscr{F} $, then $ \mathscr{F} $ is exact. As shown in \cite[Proposition 3.1]{janelidze.kelly:1994}, $ \mathscr{F} $ is a Birkhoff subcategory of $ \mathscr{C} $ if and only if the naturality square \eqref{eq:naturalitysquare} is a pushout square consisting of regular epimorphisms whenever $ f $ is a regular epimorphism. By Birkhoff's theorem, the Birkhoff subcategories of a variety of universal algebras are precisely its subvarieties. 

In \cite{janelidze.kelly:1994}, the authors consider a Galois structure $ \Gamma=(\mathscr{C},\mathscr{F},\mathsf{F},\mathsf{U},\mathscr{E},\mathscr{Z}) $, where $ \mathscr{C} $ is an exact category, $ \mathscr{F} $ is a Birkhoff subcategory of $ \mathscr{C} $ with reflection $ \mathsf{F} $ and inclusion $ \mathsf{U} $, and $ \mathscr{E} $ and $ \mathscr{Z} $ are the classes of regular epimorphisms in $ \mathscr{C} $ and $ \mathscr{F} $, respectively. They show that, if $ \mathscr{C} $ is in addition a Mal'tsev category \cite{carboni.lambek.pedicchio:1990}, i.e., any internal reflexive relation is an equivalence relation, then $ \Gamma $ is automatically admissible, and the notions of central and normal extensions coincide. Furthermore, an extension $ f:A\to B $ is trivial if and only if the pair of morphisms $ f $ and $ \eta_A $ is jointly monomorphic. This follows from \cite[Theorem 5.7]{carboni.kelly.pedicchio:1993}, where the authors show that the exact Mal'tsev categories are exactly those regular categories that satisfy the following property: given regular epimorphisms $ r:A\to B $ and $ s:A\to C $ with a common domain, their pushout exists and the induced map $ \omega $ as in the diagram
\begin{equation*}
	\begin{tikzcd}
		A \arrow[rrd, bend left=30, "s"] \arrow[ddr, bend right=30, "r"'] \arrow[rd, dotted, "\omega"]&&\\
		&B\times_D C \arrow[r, "p_2"] \arrow[d, "p_1"'] &C \arrow[d, "v"]\\
		&B \arrow[r, "u"'] &D,
	\end{tikzcd}
\end{equation*}
where the outer square is a pushout and the inner square is a pullback, is a regular epimorphism.

If we take $ \mathsf{F}:=\mathsf{Ab} $ to be the abelianization functor from the category $ \mathrm{Grp} $ of groups to its subcategory $ \mathrm{Ab} $ of abelian groups, the categorical notion of central extension exactly recovers the notion of central extension from group theory \cite{janelidze:1990}. More generally, \cite{janelidze.kelly:1994} generalizes the study of central extensions of $ \Omega $-groups with respect to a subvariety as in \cite{froehlich:1963, lue:1967, furtado-coelho:1972}. Furthermore, if $ \mathscr{C} $ is a Mal'tsev variety and $ \mathscr{F} $ is its subvariety of abelian algebras, the categorical central extensions are exactly the central extensions arising from commutator theory in universal algebra, see \cite{janelidze.kelly:2000}. 

\subsection{Central extensions in $ \mathscr{A}^\mathbb{T} $ with respect to $ \mathscr{F} $}\label{subsec:centralextensionsinA^T}

As in Section~\ref{subsec:torsiontheoryTF}, let $ \mathbb{T} $ be a finite category and $ \mathscr{A} $ be a quasi-abelian category. We have seen in Proposition~\ref{prop:torsiontheory} that $ \mathscr{F} $ is a torsion-free subcategory in $ \mathscr{A}^\mathbb{T} $. As explained in Section~\ref{subsec:categoricalGaloistheory}, this implies that $ \Gamma=(\mathscr{A}^\mathbb{T},\mathscr{F},\mathsf{F},\mathsf{U},\mathscr{E},\mathscr{Z}) $, where $ \mathsf{F} $ and $ \mathsf{U} $ are the reflection and the inclusion of $ \mathscr{F} $, respectively, and $ \mathscr{E} $ and $ \mathscr{Z} $ are the classes of regular epimorphisms in $ \mathscr{A}^\mathbb{T} $ and $ \mathscr{F} $, respectively, is an admissible Galois structure. Furthermore, the following holds: 

\begin{proposition}
	$ \mathscr{F} $ is a Birkhoff subcategory of $ \mathscr{A}^\mathbb{T} $. 
\end{proposition}

\begin{proof}
	We have already seen that $ \mathscr{F} $ is a reflective subcategory of the regular category $ \mathscr{A}^\mathbb{T} $. It remains to show that $ \mathscr{F} $ is closed under subobjects and regular quotients in $ \mathscr{A}^\mathbb{T} $. Let $ \alpha:F\to G $ be a monomorphism in $ \mathscr{A}^\mathbb{T} $ and $ G\in\mathscr{F} $. This implies that $ \alpha_T:F(T)\to 0 $ is a monomorphism for all $ T\notin\mathbb{S} $. Hence $ F(T)=0 $ for all $ T\notin\mathbb{S} $ and $ F\in\mathscr{F} $. For a regular epimorphism $ \beta:G\to H $ in $ \mathscr{A}^\mathbb{T} $ with $ G\in\mathscr{F} $, we have that $ \alpha_T:0\to G(T) $ is an epimorphism for all $ T\notin\mathbb{S} $. Thus, $ G\in\mathscr{F} $.  
\end{proof}

Assuming that the quasi-abelian category $ \mathscr{A} $ is also exact, which amounts to $ \mathscr{A} $ being abelian, we are able to characterize the trivial extensions in $ \mathscr{A}^\mathbb{T} $ with respect to $ \mathscr{F} $. 

\begin{proposition}\label{prop:trivialextensions}
	Let $ \mathscr{A} $ be abelian and $ \alpha:F\to G $ be an extension in $ \mathscr{A}^\mathbb{T} $. Then the following conditions are equivalent: 
	\begin{enumerate}
		\item $ \alpha $ is a trivial extension with respect to $ \mathscr{F} $. 
		\item $ \alpha_T $ and $ (\eta_F)_T $ are jointly monomorphic for all $ T\in\mathbb{T} $. If $ T\notin\mathbb{S} $, this amounts to $ \alpha_T $ being an isomorphism. 
	\end{enumerate}
\end{proposition}

\begin{proof}
	Since $ \mathscr{A}^\mathbb{T} $ is abelian, $ \alpha $ is a trivial extension if and only if $ \alpha $ and $ \eta_F $ are jointly monomorphic, see Section~\ref{subsec:categoricaltheoryofcentralextensions}. This is equivalent to $ \alpha_T $ and $ (\eta_F)_T $ being jointly monomorphic for all $ T\in\mathbb{T} $. If $ T\notin\mathbb{S} $, then $ (\eta_F)_T=0 $ and $ \alpha_T $ is a monomorphism. Since $ \alpha_T $ is a regular epimorphism by assumption, this implies that $ \alpha_T $ is an isomorphism.   
\end{proof}

In \cite{gran.ngahangaha:2013}, the authors show in particular that any quasi-abelian category $ \mathscr{C} $ is descent-exact \cite{duckerts-antoine:2017}, meaning that the classes of effective descent morphisms and of regular epimorphisms in $ \mathscr{C} $ coincide, i.e., for a morphism $ f:A\to B $ in $ \mathscr{C} $, the functor $ f^*:\mathscr{C}\downarrow B\to\mathscr{C}\downarrow A $ is monadic if and only if $ f $ is a regular epimorphism. More generally, they show that this property still holds true in any normal category $ \mathscr{C} $ \cite{janelidze:2010}, i.e. a pointed regular category in which any regular epimorphism is the cokernel of its kernel, such that (epimorphism, normal monomorphism)-factorizations exist. Hence \cite[2.6. Proposition]{janelidze.tholen:1994} implies that, for any regular epimorphism $ f:A\to B $ in $ \mathscr{C} $, the functor $ f^*:\mathrm{RegEpi}(B)\to\mathrm{RegEpi}(A) $ is monadic, where $ \mathrm{RegEpi}(B) $ denotes the full subcategory of the slice category $ \mathscr{C}\downarrow B $ whose objects are given by the regular epimorphisms in $ \mathscr{C} $ with codomain $ B $. We are now able to prove a characterization of the central extensions in $ \mathscr{A}^\mathbb{T} $ with respect to $ \mathscr{F} $ which is very similar to the characterization of the trivial extensions obtained in Proposition~\ref{prop:trivialextensions}. 

\begin{proposition}\label{prop:centralextensions}
	Let $ \mathscr{A} $ be a quasi-abelian category. Furthermore, let $ \alpha:F\to G $ be an extension in $ \mathscr{A}^\mathbb{T} $. Then the following conditions are equivalent: 
	\begin{enumerate}
		\item \label{it:normalextensionA^T} $ \alpha $ is a normal extension with respect to $ \mathscr{F} $.  
		\item \label{it:centralextensionA^T} $ \alpha $ is a central extension with respect to $ \mathscr{F} $. 
		\item \label{it:Ker(alpha)inF} $ \mathrm{Ker}(\alpha) $ lies in $ \mathscr{F} $. 
		\item \label{it:alpha_Tiso} $ \alpha_T $ is an isomorphism for all $ T\notin\mathbb{S} $. 
	\end{enumerate}
\end{proposition}

\begin{proof}
	The equivalence of \ref{it:normalextensionA^T}, \ref{it:centralextensionA^T} and \ref{it:Ker(alpha)inF} is clear by Proposition~\ref{prop:centralextensionsprotoadditive}. Moreover, we have that $ \mathrm{Ker}(\alpha) $ lies in $ \mathscr{F} $ if and only if $ \mathrm{Ker}(\alpha_T)=0 $ for all $ T\notin\mathbb{S} $. Since $ \mathscr{A}^\mathbb{T} $ is protomodular, this is equivalent to $ \alpha_T $ being a monomorphism. By assumption, $ \alpha_T $ is a regular epimorphism. Thus, this is equivalent to $ \alpha_T $ being an isomorphism for all $ T\notin\mathbb{S} $. 
\end{proof}

Alternatively, one could have used the results in \cite{gran.janelidze:2009} to prove the above proposition. This result is also included in Proposition~\ref{prop:highercentralextensionsprotoadditive}. 

We illustrate Propositions~\ref{prop:trivialextensions} and \ref{prop:centralextensions} in two concrete situations. In particular, we give an example of a central extension which is not trivial. Then \cite[Proposition 4.10]{gran.rossi:2007} implies that $ (\mathscr{T},\mathscr{F}) $ is not quasi-hereditary, i.e., $ \mathscr{T} $ is not closed under regular subobjects. We saw before that $ \mathscr{T} $ is closed under protosplit subobjects.

\begin{example}\label{ex:arrowcategory}
	Let $ \mathbb{T} $ be the "walking arrow" category with exactly two objects and one non-trivial morphism:
	\begin{equation*}
		\begin{tikzcd}
			1 \arrow[r] &0
		\end{tikzcd}
	\end{equation*}
	Let $ 0 $ be the only object of $ \mathbb{S} $. Note that $ 1 $ ($ 0 $) does not mean here that $ 1 $ ($ 0 $) is a terminal (initial) object. Hence $ \mathscr{A}^\mathbb{T} $ corresponds to the arrow category $ \mathrm{Arr}(\mathscr{A}) $ of $ \mathscr{A} $ and $ \mathscr{F} $ corresponds to $ \mathscr{A} $. $ \mathsf{U} $ maps an object $ A $ of $ \mathscr{A} $ to the morphism $ 0\rightarrow A $, and $ \mathsf{F} $ maps a morphism $ a:A_1\to A_0 $ to $ \mathrm{Cok}(a) $. If $ \mathscr{A} $ is abelian, an extension $ (f_0,f_1) $ in $ \mathrm{Arr}(A) $ from $ a $ to $ b $ as depicted in the diagram
	\begin{equation*}
		\begin{tikzcd}
			A_1 \arrow[d, "a"'] \arrow[r, "f_1"] &B_1 \arrow[d,"b"]\\
			A_0 \arrow[r,"f_0"'] &B_0
		\end{tikzcd}
	\end{equation*}
	is trivial if and only if $ f_1 $ is an isomorphism, and $ f_0 $ and $ \mathrm{cok}(a) $ are jointly monomorphic. If $ \mathscr{A} $ is quasi-abelian, it is central if and only if $ f_1 $ is an isomorphism. The diagram
	\begin{equation*}
		\begin{tikzcd}
			\mathbb{Z} \arrow[r, equal] \arrow[d, equal] &\mathbb{Z} \arrow[d] &\\
			\mathbb{Z} \arrow[r] &0
		\end{tikzcd}
	\end{equation*}
	in the category $ \mathrm{Ab} $ of abelian groups depicts an extension which is central but not trivial.

	It is well-known that the categories $ \mathrm{Arr}(\mathscr{A}) $, $ \mathrm{RG}(\mathscr{A}) $ of reflexive graphs, $ \mathrm{Cat}(\mathscr{A}) $ of internal categories and $ \mathrm{Grpd}(\mathscr{A}) $ of internal groupoids in $ \mathscr{A} $ are all equivalent as soon as $ \mathscr{A} $ is additive and has kernels. Given a morphism $ a:A_1\to A_0 $ in $ \mathscr{A} $, the corresponding reflexive graph is given by 
	\begin{equation*}
		\begin{tikzcd}
			A_1\oplus A_0 \arrow[r, shift left=3, "\pi_2"] \arrow[r, shift right=3, "{[}a{,}1_{A_0}{]}"'] &A_0, \arrow[l, "\iota_2"{description}]
		\end{tikzcd}
	\end{equation*}
	where $ [a,1_{A_0}] $ is the morphism induced by the universal property of the coproduct $ A_1\oplus A_0 $, and this reflexive graph has a unique internal groupoid structure. Conversely, given an internal groupoid in $ \mathscr{A} $ with underlying reflexive graph 
	\begin{equation}\label{eq:reflexivegraph}
		\begin{tikzcd}
			C_1 \arrow[r, shift left=3, "d"] \arrow[r, shift right=3, "c"'] &C_0, \arrow[l, "e"{description}]
		\end{tikzcd}
	\end{equation}
	the corresponding morphism is given by $ c\circ \mathrm{ker}(d):\mathrm{Ker}(d)\to C_0 $. Under this equivalence, $ \mathscr{F} $ corresponds to the discrete internal groupoids, i.e. those internal groupoids where $ d $ as in Diagram~\eqref{eq:reflexivegraph} is an isomorphism. In \cite{gran:1999, gran:2001}, it is shown that, if $ \mathscr{A} $ is an exact Mal'tsev category, then $ \mathrm{Grpd}(\mathscr{A}) $ is an exact Mal'tsev category. Furthermore, an extension
	\begin{equation*}
		\begin{tikzcd}
			C_1 \arrow[d, "f_1"'] \arrow[r, shift left=1, "d"] \arrow[r, shift right=1, "c"'] &C_0 \arrow[d, "f_0"]\\
			D_1 \arrow[r, shift left=1, "d"] \arrow[r, shift right=1, "c"'] &D_0
		\end{tikzcd}
	\end{equation*}
	in $ \mathrm{Grpd}(\mathscr{A}) $, i.e., $ f_0 $ and $ f_1 $ are regular epimorphisms in $ \mathscr{A} $, is a central extension with respect to the Birkhoff subcategory $ \mathrm{Discr}(\mathscr{A}) $ given by all discrete internal groupoids if and only if it is a discrete fibration, i.e., either of the above commutative squares is a pullback. In the pointed protomodular context, this is equivalent to the condition that the induced morphism between the kernels of the "domain" morphisms of the internal groupoids is an isomorphism. 
\end{example}

\begin{example}\label{ex:doublearrowcategory}
	Let $ \mathbb{T} $ be the category as in the diagram
	\begin{equation*}
		\begin{tikzcd}
			\binom{1}{1} \arrow[dr] \arrow[d] \arrow[r] &\binom{1}{0} \arrow[d]\\
			\binom{0}{1} \arrow[r] &\binom{0}{0},
		\end{tikzcd}
	\end{equation*}
	where we depicted all arrows except the identities, and $ \binom{0}{1} $ be the only object not in $ \mathbb{S} $. Then $ \mathscr{A}^\mathbb{T} $ corresponds to the double arrow category $ \mathrm{Arr}^2(\mathscr{A}) $ of $ \mathscr{A} $, $ \mathscr{F} $ corresponds to the category whose objects are given by pairs of composable morphisms $ f:A\to B $ and $ g:B\to C $ whose composite is $ 0 $, $ \mathsf{U} $ maps a corresponding pair $ (f,g) $ to the commutative square 
	\begin{equation*}
		\begin{tikzcd}
			A \arrow[d] \arrow[r, "f"] &B \arrow[d, "g"]\\
			0 \arrow[r] &C,
		\end{tikzcd}
	\end{equation*}
	and $ \mathsf{F} $ maps a commutative square
	\begin{equation*}
		\begin{tikzcd}
			A^1_1 \arrow[d, "a_1"'] \arrow[r, "a^1"] &A^1_0 \arrow[d, "a_0"]\\
			A^0_1 \arrow[r, "a^0"'] &A^0_0
		\end{tikzcd}
	\end{equation*} 
	to 
	\begin{equation*}
		\begin{tikzcd}
			A^1_1 \arrow[r, "a^1"] &A^1_0 \arrow[r, "\mathrm{cok}(a^0)a_0"] &\mathrm{Cok}(a^0). 
		\end{tikzcd}
	\end{equation*}
	If $ \mathscr{A} $ is abelian, an extension $ (f^0_0,f^0_1,f^1_0,f^1_1) $ as depicted in the diagram 
	\begin{equation*}
		\begin{tikzcd}
			A^1_1 \arrow[rr, near start, "f^1_1"] \arrow[dd,near start, "a_1"'] \arrow[dr, "a^1"] && B^1_1 \arrow[dr, "b^1"] \arrow[dd, near start, "b_1"] &\\
			&A^1_0 \arrow[dd, near end, "a_0"'] \arrow[rr, near start, "f^1_0"] &&B^1_0 \arrow[dd, near end, "b_0"]\\
			A^0_1 \arrow[rr, near start, "f^0_1"] \arrow[dr, "a^1"'] && B^0_1 \arrow[dr, "b^0"]\\
			&A^0_0 \arrow[rr, near start, "f^0_0"'] && B^0_0
		\end{tikzcd}
	\end{equation*}
	is trivial if and only if $ f^0_1 $ is an isomorphism, and $ f^0_0 $ and $ \mathrm{cok}(a^1) $ are jointly monomorphic. If $ \mathscr{A} $ is quasi-abelian, it is central if and only if $ f^0_1 $ is an isomorphism. 
\end{example} 

\section{Higher central extensions and generalized Hopf formulae in $ \mathscr{A}^\mathbb{T} $ with respect to $ \mathscr{F} $}\label{sec:highercentralextensionsinA^T}

Given a Galois structure $ \Gamma=(\mathscr{C},\mathscr{F},\mathsf{F},\mathsf{U},\mathscr{E},\mathscr{Z}) $, we denote by $ \mathrm{Ext}(\mathscr{C}) $ the full subcategory of the arrow category $ \mathrm{Arr}(\mathscr{C}) $ whose objects are given by the morphisms in $ \mathscr{E} $. It turns out that, in many cases, the full subcategory $ \mathrm{CExt}(\mathscr{C}) $ of $ \mathrm{Ext}(\mathscr{C}) $ whose objects are given by the central extensions is reflective in $ \mathrm{Ext}(\mathscr{C}) $, see e.g. \cite{janelidze.kelly:1997}. In \cite{everaert.gran.vanderlinden:2008}, the authors introduce higher central extensions to develop non-abelian homological algebra, see also \cite{everaert:2010, everaert.gran:2014, everaert:2014}. In this section, we study the higher central extensions in $ \mathscr{A}^\mathbb{T} $ with respect to $ \mathscr{F} $ and the associated generalized Hopf formulae. We start by recalling the theoretical background of this theory. 

\subsection{Higher central extensions and generalized Hopf formulae}\label{subsec:highercentralextensions}

This section is mostly based on the article \cite{duckerts-antoine:2017} since its presentation and results serve our purposes best. 

In \cite{duckerts-antoine:2017}, the author considers a \textbf{closed Galois structure} $ \Gamma=(\mathscr{C},\mathscr{F},\mathsf{F},\mathsf{U},\mathscr{E},\mathscr{Z}) $, i.e. a Galois structure as in Section~\ref{subsec:categoricalGaloistheory} such that the counit $ \varepsilon $ of the adjunction is an isomorphism and each component $ \eta_C $ of the unit of the adjunction lies in $ \mathscr{E} $. In this case, one can assume that $ \mathscr{F} $ is a full subcategory of $ \mathscr{C} $, $ \mathsf{U} $ is the inclusion and $ \varepsilon_C=1_C $ for all $ C\in\mathscr{F} $. In the following, it will not be important if $ \Gamma $ is admissible or not. The notions of trivial, normal and central extensions are defined as in Section~\ref{subsec:categoricalGaloistheory}. 

Furthermore, the author considers a pointed protomodular catgory $ \mathscr{C} $ and a class of extensions $ \mathscr{E} $ in $ \mathscr{C} $ that satisfies the following conditions: 
\begin{itemize}
	\item [(E1)] \label{it:E1} $ \mathscr{E} $ contains all isomorphisms in $ \mathscr{C} $.
	\item [(E2)] \label{it:E2} Pullbacks of morphisms in $ \mathscr{E} $ exist in $ \mathscr{C} $ and lie in $ \mathscr{E} $. 
	\item [(E3)] \label{it:E3} $ \mathscr{E} $ is closed under composition. 
	\item [(E4)] \label{it:E4} If the composite $ gf $ lies in $ \mathscr{E} $, then $ g $ lies in $ \mathscr{E} $. 
	\item [(E5)] \label{it:E5}For a commutative diagram 
	\begin{equation*}
		\begin{tikzcd}
			\mathrm{Ker}(a) \arrow[r, "\mathrm{ker}(a)"] \arrow[d, "k"'] &A_1 \arrow[r, "a"] \arrow[d, "f"'] &A_0 \arrow[d, equal]\\
			\mathrm{Ker}(b) \arrow[r, "\mathrm{ker}(b)"'] &B_1 \arrow[r, "b"] &A_0,
		\end{tikzcd}
	\end{equation*}
	the condition that $ k $ and $ a $ lie in $ \mathscr{E} $ implies that $ f $ lies in $ \mathscr{E} $. 
	\item [(M)] \label{it:M} Every morphism in $ \mathscr{E} $ is a monadic extension. 
\end{itemize}
The above conditions imply that any morphism in $ \mathscr{E} $ is a regular epimorphism or, equivalently here, a normal epimorphism. 

If $ \mathscr{C} $ is a homological category such that the classes of effective descent morphisms and of regular epimorphisms coincide, i.e., $ \mathscr{C} $ is descent-exact, and $ \mathscr{E} $ is the class of regular epimorphisms in $ \mathscr{C} $, then all the above conditions are fulfilled.

Given such a pair $ (\mathscr{C},\mathscr{E}) $, we set $ \mathrm{Ext}(\mathscr{C}):=\mathrm{Ext}_\mathscr{E}(\mathscr{C}) $, where $ \mathrm{Ext}_\mathscr{E}(\mathscr{C}) $ denotes the full subcategory of the arrow category $ \mathrm{Arr}(\mathscr{C}) $ of $ \mathscr{C} $ whose objects are given by the morphisms in $ \mathscr{E} $. A \textbf{double extension} is a morphism $ (f_0,f_1):a\to b $ in $ \mathrm{Ext}(\mathscr{C}) $ such that all morphisms in the diagram
\begin{equation*}
	\begin{tikzcd}
		A_1 \arrow[rrd, bend left=30, "f_1"] \arrow[ddr, bend right=30, "a"'] \arrow[rd, dotted]&&\\
		&A_0\times_{B_0}B_1 \arrow[r, "p_2"] \arrow[d, "p_1"'] &B_1 \arrow[d, "b"]\\
		&A_0 \arrow[r, "f_0"'] &B_0,
	\end{tikzcd}
\end{equation*}
where the inner diagram is a pullback, lie in $ \mathscr{E} $. The class of double extensions is denoted by $ \mathscr{E}^1 $. In \cite[Theorem 2.2.]{duckerts-antoine:2017}, it is shown that if $ (\mathscr{C},\mathscr{E}) $ satisfies all the conditions above, i.e., $ \mathscr{C} $ is pointed protomodular and $ \mathscr{E} $ satisfies the conditions (E1) to (E5) and (M), then $ (\mathrm{Ext}(\mathscr{C}),\mathscr{E}^1) $ satisfies the same conditions. Recursively, one gets, for every $ n\geq 1 $, a class $ \mathscr{E}^n:=(\mathscr{E}^{n-1})^1 $ of so-called \textbf{$ (n+1) $-fold extensions}, satisfying the properties (E1) to (E5) and (M) in the pointed protomodular category $ \mathrm{Ext}^n(\mathscr{C}):=\mathrm{Ext}_{\mathscr{E}^{n-1}}(\mathrm{Ext}^{n-1}(\mathscr{C})) $, where we set $ \mathscr{C}^0:=\mathscr{C} $ and $ \mathscr{E}^0:=\mathscr{E} $.

Given a closed Galois structure $ \Gamma=(\mathscr{C},\mathscr{F},\mathsf{F},\mathsf{U},\mathscr{E},\mathscr{Z}) $, we denote by $ \mathrm{NExt}(\mathscr{C}) $ the full subcategory of $ \mathrm{Ext}(\mathscr{C}) $ whose objects are given by the normal extensions with respect to $ \Gamma $. In \cite[Theorem 2.6.]{duckerts-antoine:2017}, it is assumed that $ \mathscr{C} $ is pointed protomodular, $ \mathscr{E} $ satisfies the conditions (E4), (E5) and (M), and the reflection $ \mathsf{F}:\mathscr{C}\to\mathscr{F} $ preserves pullbacks of the form
\begin{equation}\tag{A}
	\begin{tikzcd}
		A \arrow[d] \arrow[r] &B \arrow[d, "g"]\\
		D \arrow[r, "h"'] &C,
	\end{tikzcd}
\end{equation}
where $ g\in\mathscr{E} $ and $ h\in \mathrm{Split(\mathscr{E})} $. Here $ \mathrm{Split}(\mathscr{E}) $ denotes the class of morphisms which are split epimorphisms and lie in $ \mathscr{E} $. Then it is shown that there exists a closed Galois structure $ \Gamma_1=(\mathrm{Ext}(\mathscr{C}), \mathrm{NExt}(\mathscr{C}), \mathsf{F}_1, \mathsf{U}_1, \mathscr{E}^1, \mathscr{Z}^1) $ such that $ \mathrm{Ext}(\mathscr{C}) $ is pointed protomodular, $ \mathscr{E}^1 $ satisfies (E4), (E5) and (M), and $ \mathsf{F}_1 $ preserves pullbacks of form (A), where $ g $ lies in $ \mathscr{E}^1 $ and $ h $ lies in $ \mathrm{Split}(\mathscr{E}^1) $. Recursively, one gets, for every $ n\geq 1 $, a closed Galois structure $ \Gamma_n=(\mathrm{Ext}^n(\mathscr{C}),\mathrm{NExt}^n(\mathscr{C}),\mathsf{F}_n,\mathsf{U}_n,\mathscr{E}^n,\mathscr{Z}^n) $, where $ \mathrm{NExt}^n(\mathscr{C}) $ is the full subcategory of $ \mathrm{Ext}^n(\mathscr{C}) $ whose objects are given by the $ n $-fold extensions which are normal with respect to $ \Gamma_{n-1} $, where we set $ \Gamma_0:=\Gamma $. We denote by $ \eta^n $ the unit of the adjunction in $ \Gamma_n $. 

We recall that a semi-abelian category in the sense of \cite{janelidze.marki.tholen:2002} is a category which is finitely complete, finitely cocomplete, pointed, exact and protomodular. In \cite[Lemma 4.4.]{everaert.gran.vanderlinden:2008}, it is in particular shown that any semi-abelian category $ \mathscr{C} $ with Birkhoff subcategory $ \mathscr{F} $, and $ \mathscr{E} $ and $ \mathscr{Z} $ being the classes of regular epimorphisms in $ \mathscr{C} $ and $ \mathscr{F} $, respectively, yields a Galois structure $ \Gamma=(\mathscr{C},\mathscr{F},\mathsf{F},\mathsf{U},\mathscr{E},\mathscr{Z}) $ that satisfies all the conditions required above. Moreover, \cite[Proposition 4.5]{everaert.gran.vanderlinden:2008} shows that $ \mathrm{NExt}^n(\mathscr{C})=\mathrm{CExt}^n(\mathscr{C}) $, where $ \mathrm{CExt}^n(\mathscr{C}) $ denotes the full subcategory of $ \mathrm{Ext}^n(\mathscr{C}) $ whose objects are given by the $ n $-fold extensions which are central with respect to $ \Gamma_{n-1} $, see also \cite{everaert:2014}. 

It is easy to see that a protoadditive functor into a protomodular category preserves pullbacks of form (A), where $ h $ is a split epimorphism. In the following, we will get a similar result to Proposition~\ref{prop:centralextensionsprotoadditive} for higher extensions. 

Given an object $ A\in\mathscr{C} $, we denote by $ [A] $ the kernel of $ \eta_A:A\to\mathsf{F}(A) $. For $ n\geq 1 $ and $ A\in \mathrm{Ext}^n(\mathscr{C}) $, we denote by $ [A]^n $ the kernel of $ \eta^n_A:A\to\mathsf{F}^n(A) $. This defines a functor $ [-]^n:\mathrm{Ext}^n(\mathscr{C})\to\mathrm{Ext}^n(\mathscr{C}) $. Let us consider the natural numbers by their standard (von Neumann) construction, and put $ 0:=\emptyset $ and $ n:=\{0,\ldots,n-1\} $ for $ n\geq 1 $. Let $ n\geq 0 $. We denote by $ \mathcal{P}(n) $ the partial order given by the power set of $ n $ and view it is a category. Let us consider the functor category $ \mathscr{C}^{\mathcal{P}(n)^{\mathrm{op}}} $. For an object $ A $ in $ \mathscr{C}^{\mathcal{P}(n)^{\mathrm{op}}} $, we set, for any $ S\subseteq T\subseteq n $, $ A_S:=A(S) $ and $ a^T_S:=A(S\subseteq T) $, which is an arrow in $ \mathscr{C} $ from $ A_T $ to $ A_S $. We define $ a_i:=a^{n}_{n\setminus\{i\}} $ for any $ 0\leq i\leq n-1 $. We write $ (A_S)_{S\subseteq n} $ synonymously for $ A $. For $ n\geq 1 $, the functor $ \delta_{n-1}:\mathscr{C}^{\mathcal{P}(n)^{\mathrm{op}}}\to \mathrm{Arr}(\mathscr{C}^{\mathcal{P}(n-1)^{\mathrm{op}}}) $ which maps a morphism $ f:A\to B $ in $ \mathscr{C}^{\mathcal{P}(n)^{\mathrm{op}}} $ to the commutative square 
\begin{equation*}
	\begin{tikzcd}
		(A_{S\cup\{n-1\}})_{S\subseteq n-1} \arrow[d, "(a^{S\cup \{n-1\}}_{S})_{S\subseteq n-1}"'] \arrow[rrr, "(f_{S\cup\{n-1\}})_{S\subseteq n-1}"] &&&(B_{S\cup\{n-1\}})_{S\subseteq n-1} \arrow[d, "(b^{S\cup \{n-1\}}_{S})_{S\subseteq n-1}"]\\
		(A_{S})_{S\subseteq n-1} \arrow[rrr, "(f_{S})_{S\subseteq n-1}"'] &&&(B_{S})_{S\subseteq n-1}\\
	\end{tikzcd}
\end{equation*} 
defines an isomorphism of categories. Hence the composite $ \phi:=\mathrm{Arr}^{n-1}(\delta_0)\cdots \mathrm{Arr}(\delta_{n-2})\delta_{n-1} $ yields an isomorphism between $ \mathscr{C}^{\mathcal{P}(n)^{\mathrm{op}}} $ and $ \mathrm{Arr}^n(\mathscr{C}) $. We set $ \underline{\mathrm{Ext}}^n(\mathscr{C}):=\phi^{-1}(\mathrm{Ext}^n(\mathscr{C})) $. Even though $ \phi $ is just one possible choice of isomorphism between $ \mathscr{C}^{\mathcal{P}(n)^{\mathrm{op}}} $ and $ \mathrm{Arr}^n(\mathscr{C}) $, it is shown in \cite[Proposition 1.16.]{everaert.goedecke.vanderlinden:2012} that $ \underline{\mathrm{Ext}}^n(\mathscr{C}) $ is the full subcategory of $ \mathscr{C}^{\mathcal{P}(n)^{\mathrm{op}}} $ whose objects are given by the functors $ \mathcal{P}(n)^{\mathrm{op}}\to\mathscr{C} $ such that the limit $ \mathrm{lim}_{J\subsetneq I} A_J $ exists for all $ \emptyset\neq I\subseteq n $ and the induced morphism $ A_I\to \mathrm{lim}_{J\subsetneq I} A_J $ lies in $ \mathscr{E} $. In the following, we will not distinguish between $ \mathrm{Ext}^n(\mathscr{C}) $ and $ \underline{\mathrm{Ext}}^n(\mathscr{C}) $ and assume the application of $ \phi $ implicitly. We define $ \iota^n:\mathscr{C}\to\mathrm{Ext}^n(\mathscr{C}) $ to be the functor which maps an object $ A $ in $ \mathscr{C} $ to the functor which maps $ n $ to $ A $, and $ S $ to $ 0 $ for any $ S\subsetneq n $. In \cite[Theorem 2.15.]{duckerts-antoine:2017}, it is shown that the functor $ [-]^n:\mathrm{Ext}^n(\mathscr{C})\to\mathrm{Ext}^n(\mathscr{C}) $ factors through $ \iota^n $, i.e., there exists a functor $ [-]_n:\mathrm{Ext}^n(\mathscr{C})\to\mathscr{C} $ such that $ [A]^n=\iota^n([A]_n) $ for all $ A\in\mathrm{Ext}^n(\mathscr{C}) $. 

\begin{example}\label{ex:doublecentralextensionsofgroups}
	Let us consider the Galois structure $ \Gamma=(\mathrm{Grp},\mathrm{Ab},\mathsf{Ab},\mathsf{U},\mathscr{E},\mathscr{Z}) $, where $ \mathsf{Ab}:\mathrm{Grp}\to\mathrm{Ab} $ is the abelianization functor, and $ \mathscr{E} $ and $ \mathscr{Z} $ are the classes of regular epimorphisms, i.e. surjective group homomorphisms, in $ \mathrm{Grp} $ and $ \mathrm{Ab} $, respectively. It is clear that $ \mathrm{Grp} $ is a semi-abelian category with Birkhoff subcategory $ \mathrm{Ab} $. The centralization of an extension $ f:A\to B $ is given by the induced morphism from the quotient of $ A $ by its normal subgroup $ [\mathrm{Ker}(f),A] $ to $ B $: 
	\begin{equation*}
		\begin{tikzcd}
			A \arrow[dr] \arrow[rr, "f"] &&B\\
			&A/[\mathrm{Ker}(f),A] \arrow[ur, dotted]
		\end{tikzcd}
	\end{equation*}
	In \cite{janelidze:1991}, it is shown that a double extension, i.e. a commutative square
	\begin{equation*}
		\begin{tikzcd}
			A_1 \arrow[d, "a"'] \arrow[r, "f_1"] &B_1 \arrow[d,"b"]\\
			A_0 \arrow[r,"f_0"'] &B_0
		\end{tikzcd}
	\end{equation*}
	of surjective group homomorphisms such that the induced morphism from $ A_1 $ to the pullback of $ f_0 $ along $ b $ is surjective as well, is central if and only if the conditions $ [\mathrm{Ker}(a)\cap\mathrm{Ker}(f_1),A]=0 $ and $ [\mathrm{Ker}(a),\mathrm{Ker}(f_1)]=0 $ hold. The centralisation of a double extension is given by the induced square 
	\begin{equation*}
		\begin{tikzcd}
			A_1/C \arrow[d, dotted] \arrow[r, dotted] &B_1 \arrow[d,"b"]\\
			A_0 \arrow[r,"f_0"'] &B_0,
		\end{tikzcd}
	\end{equation*}
	where $ C $ denotes the product $ [\mathrm{Ker}(a)\cap\mathrm{Ker}(f_1),A]\cdot[\mathrm{Ker}(a),\mathrm{Ker}(f_1)] $ of the groups $ [\mathrm{Ker}(a)\cap\mathrm{Ker}(f_1),A] $ and $ [\mathrm{Ker}(a),\mathrm{Ker}(f_1)] $. 
\end{example}

An object $ P $ in $ \mathscr{C} $ is called \textbf{$ \mathscr{E} $-projective} if in any diagram 
\begin{equation*}
	\begin{tikzcd}
		&P \arrow[d, "p"] \arrow[dl, dotted, "p'"']\\
		A \arrow[r, "f"'] &B,
	\end{tikzcd}
\end{equation*}
where $ f $ is an extension, there exists a morphism $ p' $ such that $ fp'=p $. If $ \mathscr{E} $ is the class of regular epimorphisms, we speak of regular projective objects. An \textbf{$ \mathscr{E} $-projective presentation} of an object $ A $ in $ \mathscr{C} $ is given by an extension $ p:P\to A $ where $ P $ is an $ \mathscr{E} $-projective object. We say that $ \mathscr{C} $ \textbf{has enough $ \mathscr{E} $-projective objects} if any object in $ \mathscr{C} $ admits an $ \mathscr{E} $-projective presentation. An \textbf{$ n $-fold $ \mathscr{E} $-projective presentation} of an object $ A $ in $ \mathscr{C} $ is given by an object $ P $ in $ \mathrm{Ext}^n(\mathscr{C}) $ with $ P_0=A $ and $ P_S $ an $ \mathscr{E} $-projective object for all $ \emptyset\neq S\subseteq n $. In \cite[Lemma 2.13.]{duckerts-antoine:2017}, it is shown that, if $ \mathscr{C} $ has enough $ \mathscr{E} $-projective objects, then any $ A $ in $ \mathscr{C} $ admits at least one $ n $-fold $ \mathscr{E} $-projective presentation. In \cite[Theorem 3.13.]{duckerts-antoine:2017}, which computes higher Galois groups, it is in particular shown that the so-called \textbf{Hopf formula for the $ (n+1) $-st homology} of $ A $ with respect to $ \mathscr{F} $
\begin{equation*}
	H_{n+1}(A,\mathscr{F}):=\frac{[P_n]\cap\bigcap_{0\leq i\leq n-1}\mathrm{Ker}(p_i)}{[P]_n},
\end{equation*}
where $ P $ is an $ n $-fold $ \mathscr{E} $-projective presentation of an object $ A $ in $ \mathscr{C} $, does not depend on the chosen presentation. Furthermore, $ H_{n+1}(A,\mathscr{F}) $ lies in $ \mathscr{F} $. In \cite{everaert.gran.vanderlinden:2008}, it is proven that the functors $ H_{n+1}(-,\mathscr{F}) $ coincide with the Barr-Beck left derived functors of the reflection $ \mathsf{F}:\mathscr{C}\to\mathscr{F} $ whenever $ \mathscr{C} $ is a semi-abelian category which is monadic over the category $ \mathrm{Set} $ of sets, $ \mathscr{F} $ is a Birkhoff subcategory of $ \mathscr{C} $, and $ \mathscr{E} $ and $ \mathscr{Z} $ are the classes of regular epimorphisms in $ \mathscr{C} $ and $ \mathscr{F} $, respectively. 

\begin{example}
	Let us again consider $ \mathscr{C} $ to be the category $ \mathrm{Grp} $ of groups and $ \mathscr{F} $ to be its subcategory $ \mathrm{Ab} $ of abelian groups as in Example~\ref{ex:doublecentralextensionsofgroups}. We know that $ \mathrm{Grp} $ has enough regular projective objects. Let $ p:P\to A $ be a regular projective presentation of any group $ A $. For example, $ P $ can be taken to be the free group on the underlying set of $ A $. Then the second Hopf formula $ H_2(A,\mathrm{Ab}) $ is given by 
	\begin{equation*}
		\frac{[P,P]\cap\mathrm{Ker}(p)}{[\mathrm{Ker}(p),P]}. 
	\end{equation*}
	It is shown in \cite{hopf:1942} that this expresion is exactly the second integral homology group of $ A $. If 
	\begin{equation*}
		\begin{tikzcd}
			P \arrow[d, "p_1"'] \arrow[r, "p_2"] &P_2 \arrow[d]\\
			P_1 \arrow[r] &A
		\end{tikzcd}
	\end{equation*}
	is a double regular projective presentation of $ A $, i.e., it is a double extension and $ P,P_1,P_2 $ are regular projective objects, the third Hopf formula $ H_3(A,\mathrm{Ab}) $ is given by 
	\begin{equation*}
		\frac{[P,P]\cap\mathrm{Ker}(p_1)\cap\mathrm{Ker}(p_2)}{[\mathrm{Ker}(p_1)\cap\mathrm{Ker}(p_2),P]\cdot[\mathrm{Ker}(p_1),\mathrm{Ker}(p_2)]}
	\end{equation*}
	and coincides with the third integral homology group of $ A $. More generally, it is proven in \cite{brown.ellis:1988} that the $ n $-th homology group of $ A $ is given by $ H_n(A,\mathrm{Ab}) $. 
\end{example}

The following result follows from \cite[Theorem 3.16.]{duckerts-antoine:2017}, see also \cite[4.4. Torsion theories]{duckerts-antoine:2017}. 
\begin{proposition}\label{prop:highercentralextensionsprotoadditive}
	Let $ \Gamma=(\mathscr{C},\mathscr{F},\mathsf{F},\mathsf{U},\mathscr{E},\mathscr{Z}) $ be a closed Galois structure, where $ \mathscr{C} $ is a pointed protomodular category and $ \mathscr{E} $ satisfies the conditions (E4), (E5) and (M). Furthermore, we assume that $ \mathsf{F} $ is a protoadditive functor. Let $ n\geq 1 $, $ A\in\mathrm{Ext}^n(\mathscr{C}) $, $ B\in\mathscr{C} $ and $ P\in\mathrm{Ext}^n(\mathscr{C}) $ be an $ n $-fold $ \mathscr{E} $-projective presentation of $ B $. 
	\begin{enumerate}
		\item The following conditions are equivalent: 
		\begin{enumerate}
			\item $ A\in\mathrm{NExt}^n(\mathscr{C}) $. 
			\item $ A\in\mathrm{CExt}^n(\mathscr{C}) $. 
			\item $ \bigcap_{0\leq i\leq n-1}\mathrm{Ker}(a_i)\in\mathscr{F} $. 
		\end{enumerate}
		\item We have that
		\begin{equation*}
			[A]_n=[\bigcap_{0\leq i\leq n-1}\mathrm{Ker}(a_i)]. 
		\end{equation*}
		\item The Hopf formula for the $ (n+1) $-st homology of $ B $ with respect to $ \mathscr{F} $ is given by 
		\begin{equation*}
			H_{n+1}(B,\mathscr{F}):=\frac{[P_n]\cap\bigcap_{0\leq i\leq n-1}\mathrm{Ker}(p_i)}{[\bigcap_{0\leq i\leq n-1}\mathrm{Ker}(p_i)]}. 
		\end{equation*}
	\end{enumerate}
\end{proposition}

\subsection{Higher central extensions and generalized Hopf formulae in $ \mathscr{A}^\mathbb{T} $ with respect to $ \mathscr{F} $}\label{subsec:highercentralextensionsinA^T}

After we studied the central extensions in $ \mathscr{A}^\mathbb{T} $ with respect to $ \mathscr{F} $ in Section~\ref{subsec:centralextensionsinA^T}, we apply the theory recalled in the previous section to obtain a characterization of the higher central extensions in $ \mathscr{A}^\mathbb{T} $. Indeed, $ \Gamma=(\mathscr{A}^\mathbb{T},\mathscr{F},\mathsf{F},\mathsf{U},\mathscr{E},\mathscr{Z}) $ is a closed Galois structure, $ \mathscr{A}^\mathbb{T} $ is pointed protomodular, $ \mathscr{E} $ satisfies the conditions (E4), (E5) and (M) since $ \mathscr{A} $ is assumed to be quasi-abelian, and the reflection $ \mathsf{F} $ is protoadditive as it is a left-adjoint between additive categories. 

Let $ n\geq 1 $ and $ A\in\mathrm{Ext}^n(\mathscr{A}^\mathbb{T}) $. As explained in Section~\ref{subsec:highercentralextensions}, we denote, for all $ 0\leq i\leq n-1  $, by $ a_i $ the image of the morphism $ n\setminus\{i\}\subseteq n $ seeing $ A $ as an object in $ (\mathscr{A}^\mathbb{T})^{\mathcal{P}(n)^\mathrm{op}} $. In addition to Proposition~\ref{prop:centralextensions}, Proposition~\ref{prop:highercentralextensionsprotoadditive} implies the following: 

\begin{proposition}\label{prop:highercentralextensions}
	Let $ n\geq 1 $ and $ \mathscr{A} $ be a quasi-abelian category. The following conditions are equivalent: 
	\begin{enumerate}
		\item $ A\in\mathrm{NExt}^n(\mathscr{A}^\mathbb{T}) $. 
		\item $ A\in\mathrm{CExt}^n(\mathscr{A}^\mathbb{T}) $.
		\item $ \bigcap_{0\leq i\leq n-1}\mathrm{Ker}(a_i)\in\mathscr{F} $.
		\item $ \{(a_i)_T\}_{0\leq i\leq n-1} $ are jointly monomorphic for all $ T\notin\mathbb{S} $. 
	\end{enumerate}
\end{proposition}
\begin{proof}
	We note that $ \bigcap_{0\leq i\leq n-1}\mathrm{Ker}(a_i)\in\mathscr{F} $ is equivalent to $ \bigcap_{0\leq i\leq n-1} \mathrm{Ker}((a_i)_T)=0 $ for all $ T\notin\mathbb{S} $. This means that $ \{(a_i)_T\}_{0\leq i\leq n-1} $ are jointly monomorphic for all $ T\notin\mathbb{S} $. 
\end{proof}

We continue the study of the Examples~\ref{ex:arrowcategory} and \ref{ex:doublearrowcategory}. 

\begin{example}
	As in Example~\ref{ex:arrowcategory}, let us consider $ \mathbb{T} $ to be walking arrow category $ 1\to 0 $ and $ \mathbb{S} $ to be the full subcategory of $ \mathbb{T} $ which contains only $ 0 $ as an object. We saw that an extension $ (f_0,f_1) $ in $ \mathrm{Arr}(\mathscr{A}) $ from $ a $ to $ b $ as depicted in the diagram
	\begin{equation*}
		\begin{tikzcd}
			A_1 \arrow[d, "a"'] \arrow[r, "f_1"] &B_1 \arrow[d,"b"]\\
			A_0 \arrow[r,"f_0"'] &B_0
		\end{tikzcd}
	\end{equation*}
	is central if and only if $ f_1 $ is an isomorphism.
	We note that, for an object $ a:A_1\to A_0 $ of $ \mathrm{Arr}(\mathscr{A}) $, $ [a] $ is given by the epimorphism part $ \varphi\circ\mathrm{cok}(\mathrm{ker}(a)):A_1\to\mathrm{Ker}(\mathrm{cok}(a)) $ of the (epimorphism, normal monomorphism)-factorization of $ a $, see also the beginning of Section~\ref{subsec:torsiontheoryTF}. We set $ \mathrm{Coim}(a):=\mathrm{Ker}(\mathrm{cok}(a)) $.  The centralization of the extension $ (f_0,f_1) $ is given by the induced morphism in the commutative diagram
	\begin{equation*}
		\begin{tikzcd}
			\mathrm{Ker}(f_1) \arrow[dd] \arrow[r, equal] &\mathrm{Ker}(f_1) \arrow[dd, "a|_{\mathrm{Ker}(f_1)}"'] \arrow[r] &A_1 \arrow[rr, near start, "f_1"] \arrow[dr] \arrow[dd, "a"'] &&B_1 \arrow[dd, "b"]\\
			&&&\frac{A_1}{\mathrm{Ker}(f_1)} \arrow[dd] \arrow[ru, dotted]&\\
			\mathrm{Coim}(a|_{\mathrm{Ker}(f_1)}) \arrow[r] &\mathrm{Ker}(f_0) \arrow[r] &A_0 \arrow[rd] \arrow[rr, near start, "f_0"'] && B_0.\\
			&&&\frac{A_0}{\mathrm{Coim}(a|_{\mathrm{Ker}(f_1)})} \arrow[ru, dotted]
		\end{tikzcd}
	\end{equation*}
	If $ (f_0,f_1) $ yields a regular projective presentation for $ (B_0,B_1) $, then the Hopf formula for the second homology is given by 
	\begin{equation*}
		H_2(b,\mathscr{A})=\frac{\mathrm{Coim}(a)\cap\mathrm{Ker}(f_0)}{\mathrm{Coim}(a|_{\mathrm{Ker}(f_1)})}. 
	\end{equation*}
	A double extension as the outer part of the diagram
	
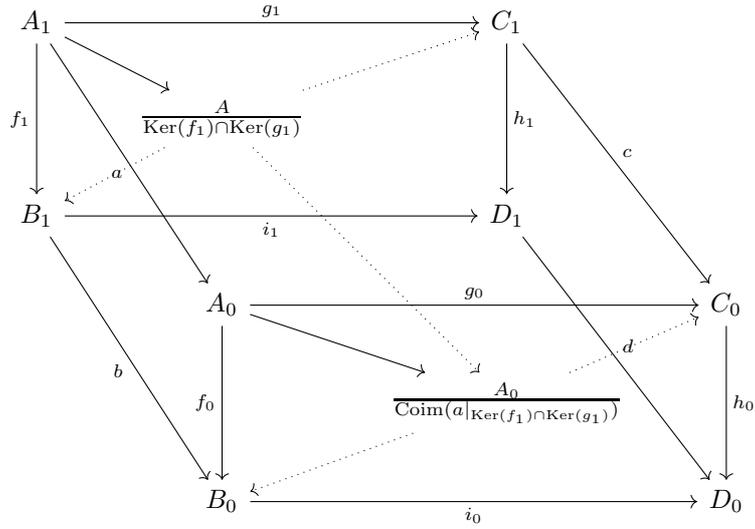
\begin{equation*}
		\begin{tikzcd}
			A_1 \arrow[rr, "g_1"] \arrow[dd, "f_1"'] \arrow[dr] \arrow[dddr, "a"'] &&C_1 \arrow[dd, "h_1"] \arrow[dddr, "c"]&\\
			&\frac{A}{\mathrm{Ker}(f_1)\cap\mathrm{Ker}(g_1)} \arrow[dddr, dotted] \arrow[ru, dotted] \arrow[ld, dotted]&&\\
			B_1 \arrow[dddr, "b"'] \arrow[rr, "i_1"'] && D_1 \arrow[dddr, "d"] &\\
			& A_0 \arrow[dr] \arrow[rr, "g_0"] \arrow[dd, "f_0"'] &&C_0 \arrow[dd, "h_0"]\\
			&&\frac{A_0}{\mathrm{Coim}(a|_{\mathrm{Ker}(f_1)\cap\mathrm{Ker}(g_1)})} \arrow[dl, dotted] \arrow[ru, dotted]&\\
			&B_0 \arrow[rr, "i_0"'] &&D_0
		\end{tikzcd}
	\end{equation*}
	is central if and only if $ f_1 $ and $ g_1 $ are jointly monomorphic. Its centralization is given via the induced morphisms depicted in the above diagram. If the double extension yields a double regular projective presentation for $ d:D_1\to D_0 $, then the Hopf formula for the third homology is given by 
	\begin{equation*}
		H_3(d,\mathscr{A})=\frac{\mathrm{Coim}(a)\cap\mathrm{Ker}(f_1)\cap\mathrm{Ker}(g_1)}{\mathrm{Coim}(a|_{\mathrm{Ker}(f_1)\cap\mathrm{Ker}(g_1)})}. 
	\end{equation*}
	It is clear how this formula generalizes to arbitrary $ n\geq2 $. 
\end{example}

\begin{example}
	Let $ \mathbb{T} $ and $ \mathbb{S} $ be as in Example~\ref{ex:doublearrowcategory}. We saw that an extension $ (f^0_0, f^0_1, f^1_0, f^1_1) $ in $ \mathrm{Arr}^2(\mathscr{A}) $ from $ \mathbb{A} $ to $ \mathbb{B} $ as depicted in the diagram 
	 
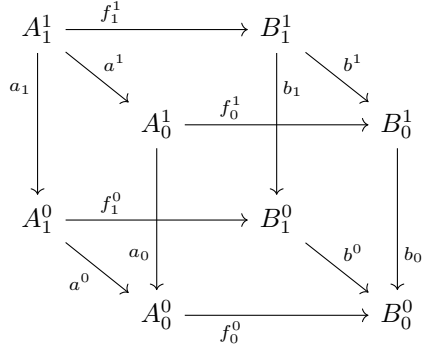
\begin{equation*}
	 	\begin{tikzcd}
	 		A^1_1 \arrow[rr, near start, "f^1_1"] \arrow[dd,near start, "a_1"'] \arrow[dr, "a^1"] && B^1_1 \arrow[dr, "b^1"] \arrow[dd, near start, "b_1"] &\\
	 		&A^1_0 \arrow[dd, near end, "a_0"'] \arrow[rr, near start, "f^1_0"] &&B^1_0 \arrow[dd, near end, "b_0"]\\
	 		A^0_1 \arrow[rr, near start, "f^0_1"] \arrow[dr, "a^0"'] && B^0_1 \arrow[dr, "b^0"]\\
	 		&A^0_0 \arrow[rr, near start, "f^0_0"'] && B^0_0
	 	\end{tikzcd}
	 \end{equation*}
 	is central if and only if $ f^0_1 $ is an isomorphism. We observe that $ [\mathbb{A}] $ is given by 
 	\begin{equation*}
 		\begin{tikzcd}
 			0 \arrow[d] \arrow[rr, equal] &&0 \arrow[d]\\
 			A^0_1 \arrow[rr, "\varphi\circ\mathrm{cok}(\mathrm{ker}(a^0))"'] &&\mathrm{Coim}(a^0). 
 		\end{tikzcd}
 	\end{equation*} 
 	Hence the centralization of $ (f^0_0, f^0_1, f^1_0, f^1_1) $ is given by the following extension: 
 	\begin{equation*}
 		\begin{tikzcd}
 			A^1_1 \arrow[rr, near start] \arrow[dd,near start] \arrow[dr] && B^1_1 \arrow[dr] \arrow[dd, near start] &\\
 			&A^1_0 \arrow[dd, near end] \arrow[rr, near start] &&B^1_0 \arrow[dd, near end]\\
 			\frac{A^0_1}{\mathrm{Ker}(f^0_1)} \arrow[rr, near start] \arrow[dr] && B^0_1 \arrow[dr]\\
 			&\frac{A^0_0}{\mathrm{Coim}(a^0|_{\mathrm{Ker}(f^0_1)})} \arrow[rr, near start] && B^0_0.
 		\end{tikzcd}
 	\end{equation*}
 	Moreover, if $ (f^0_0,f^0_1,f^1_0,f^1_1) $ is a regular projective presentation of $ \mathbb{B} $, the Hopf formula for the second homology is given by 
 	\begin{equation*}
 		\begin{tikzcd}
 			H_2(\mathbb{B},2\textrm{-Arr}(\mathscr{A}))=\Bigl(0 \arrow[r, equal] &0 \arrow[r] &\frac{\mathrm{Coim}(a^0)\cap\mathrm{Ker}(f^0_0)}{\mathrm{Coim}(a^0|_{\mathrm{Ker}(f^0_1)})}\Bigr). 
 		\end{tikzcd}
 	\end{equation*}
 	It is easy to see how this generalizes for higher central extensions. 
\end{example}

In \cite[Theorem 4.11]{everaert.gran:2014}, it is shown that a torsion theory $ (\mathscr{T},\mathscr{F}) $ in a homological category $ \mathscr{C} $, where $ \mathsf{F} $ is protoadditive and the composite $ \mathrm{ker}(f)\varepsilon_{\mathrm{Ker}(f)}:\mathsf{T}(\mathrm{Ker}(f))\to A $ is a normal monomorphism for any morphism $ f:A\to B $ in $ \mathscr{C} $, induces, for any $ n\geq1 $, a torsion theory $ (\mathscr{T}_n,\mathscr{F}_n) $ in the category $ \mathrm{Ext}^n(\mathscr{C}) $ of $ n $-fold extensions with respect to the class of extensions in $ \mathscr{C} $ consisting of all normal epimorphisms in $ \mathscr{C} $. The consequences of this result in our setting will be studied in a future article.


\begin{thebibliography}{10}
	
\bibitem{barr.grillet.vanosdol:1971}
\textsc{M. Barr, P.A. Grillet and D.H. van Osdol}, Exact categories and categories of sheaves, \textit{Lect. Notes Math., vol. 236, Springer-Verlag} (1971). 	
	
\bibitem{bonet.dierolf:2005}
\textsc{J. Bonet and S. Dierolf}, The pullback for bornological and ultrabornological spaces, \textit{Note Mat.} \textbf{25} 1 (2005/06), 63--67.

\bibitem{borceux:1994}
\textsc{F. Borceux}, Handbook of categorical algebra, Volume II: Categories and structures, \textit{Encyclopedia Math. Appl., vol. 51, Cambridge Univ. Press} (1994). 

\bibitem{bourn:1991}
\textsc{D. Bourn}, Normalization equivalence, kernel equivalence and affine categories, \textit{Lect. Notes Math., vol. 1488, Springer-Verlag} (1971), 43--62. 

\bibitem{brown.ellis:1988}
\textsc{R. Brown and G.J. Ellis}, Hopf formulae for the higher homology of a group, \textit{Bull. London Math. Soc.} \textbf{20} (1988), 124--128.

\bibitem{carboni.kelly.pedicchio:1993}
\textsc{A. Carboni, G.M. Kelly and M.C. Pedicchio}, Some remarks on Maltsev and Goursat categories, \textit{Appl. Categorical Struct.} \textbf{1} (1993), 385--421.

\bibitem{carboni.lambek.pedicchio:1990}
\textsc{A. Carboni, J. Lambek and M. C. Pedicchio}, Diagram chasing in Mal'cev categories, \textit{J. Pure Appl. Algebra} \textbf{69} (1990), 271--284.

\bibitem{cassidy.hebert.kelly:1985}
\textsc{C. Cassidy, M. H\'{e}bert, G.M. Kelly}, Reflective subcategories, localizations and factorization systems, \textit{J. Austral. Math. Soc. (Series A)} \textbf{38} (19985), 287--329. 

\bibitem{duckerts-antoine:2017}
\textsc{M. Duckerts-Antoine}, Fundamental group functors in descent-exact homological categories, \textit{Adv. in Mathematics}, \textbf{310} (2017), 64--120.  

\bibitem{everaert:2010}
\textsc{T. Everaert}, Higher central extensions and Hopf formulae, \textit{J. Algebra} \textbf{324} (2010), 1771--1789.

\bibitem{everaert:2014}
\textsc{T. Everaert}, Higher central extensions in Mal'tsev categories, \textit{Appl. Categor. Struct.} \textbf{22} (2014), 961--979. 

\bibitem{everaert.goedecke.vanderlinden:2012}
\textsc{T. Everaert, J. Goedecke and T. Van der Linden}, Resolutions, higher extensions and the relative Mal'tsev axiom, \textit{J. Algebra} \textbf{371} (2012), 132--155. 

\bibitem{everaert.gran:2010}
\textsc{T. Everaert and M. Gran}, Homology of n-fold groupoids, \textit{Theory Appl. Categ.} \textbf{23} (2010), 22--41.

\bibitem{everaert.gran:2014}
\textsc{T. Everaert and M. Gran}, Protoadditive functors, derived torsion theories and homology, \textit{J. Pure Appl. Algebra} \textbf{219} 8 (2014), 3629--3676.

\bibitem{everaert.gran.vanderlinden:2008}
\textsc{T. Everaert, M. Gran and T. Van der Linden}, Higher Hopf formulae for homology via Galois theory, \textit{Adv. in Mathematics} \textbf{217} (2008), 2231--2267.

\bibitem{froehlich:1963}
\textsc{A. Fröhlich}, Baer-invariants of algebras, \textit{Trans. Amer. Math. Soc.} \textbf{109} (1963), 221--244.

\bibitem{furtado-coelho:1972}
\textsc{J. Furtado-Coelho}, Varieties of $ \Omega $-groups and associated functors, Ph.D. Thesis, King's College, University of London (1972).

\bibitem{gran:1999}
\textsc{M. Gran}, Internal categories in Mal'cev categories, \textit{J. Pure Appl. Algebra} \textbf{143} (1999), 221--229.

\bibitem{gran:2001}
\textsc{M. Gran}, Central extensions and internal groupoids in Maltsev categories, \textit{J. Pure Appl. Algebra} \textbf{155} (2001), 139--166.

\bibitem{gran.janelidze:2009}
\textsc{M. Gran and G. Janelidze}, Covering morphisms and normal extensions in Galois structures associated with torsion theories, \textit{Cah. Topol. G\'eom. Diff\'er. Cat\'eg.} \textbf{50} 3 (2009), 171--188.

\bibitem{gran.ngahangaha:2013}
\textsc{M. Gran and O. Ngaha Ngaha}, Effective descent morphisms in star-regular categories, \textit{Homol. Homotopy Appl.} \textbf{15} 2 (2013), 127--144.

\bibitem{gran.rossi:2007}
\textsc{M. Gran and V. Rossi}, Torsion theories and coverings of topological groups, \textit{J. Pure Appl. Algebra} \textbf{208} (2007), 135-–151. 

\bibitem{gruson:1966}
\textsc{L. Gruson}, Compl\'{e}tion ab\'{e}lienne, \textit{Bull. Sci. Math.} \textbf{90} (1966), 17--40. 
 
\bibitem{hopf:1942}
\textsc{H. Hopf}, Fundamentalgruppe und zweite Bettische Gruppe, \textit{Comment. Math. Helv.} \textbf{14} (1942), 257--309.

\bibitem{janelidze:1989}
\textsc{G. Janelidze}, The fundamental theorem of Galois theory, \textit{Math. USSR Sbornik} \textbf{64} 2 (1989), 359--374.

\bibitem{janelidze:1990}
\textsc{G. Janelidze}, Pure Galois theory in categories, \textit{J. Algebra} \textbf{132} (1990), 270--286.

\bibitem{janelidze:1991}
\textsc{G. Janelidze}, Precategories and Galois theory, \textit{Lecture Notes in Math.} \textbf{1488}, Springer (1991), 157--173.

\bibitem{janelidze:2003}
\textsc{G. Janelidze}, Internal crossed modules, \textit{Georgian Math. J.} \textbf{10} (2003), 99--114.

\bibitem{janelidze:2010}
\textsc{Z. Janelidze}, The pointed subobject functor, 3×3 lemmas, and subtractivity of spans, \textit{Theory Appl. Categ.} \textbf{23} (2010), 221--242. 

\bibitem{janelidze.kelly:1994}
\textsc{G. Janelidze and G.M. Kelly}, Galois theory and a general notion of central extension, \textit{J. Pure Appl. Algebra} \textbf{2} (1994), 135--161.

\bibitem{janelidze.kelly:1997}
\textsc{G. Janelidze and G.M. Kelly}, The reflectiveness of covering morphisms in algebra and geometry, \textit{Theory Apppl. Categ.} \textbf{3} 6 (1997), 132--159. 

\bibitem{janelidze.kelly:2000}
\textsc{G. Janelidze and G.M. Kelly}, Central extensions in Mal'tsev varieties, \textit{Theory Appl. Categ.} \textbf{7} 10 (2000), 219--226.

\bibitem{janelidze.marki.tholen:2002}
\textsc{G. Janelidze, L. Márki and W. Tholen}, Semi-abelian categories, \textit{J. Pure Appl. Algebra} \textbf{168} (2002), 367--386.

\bibitem{janelidze.tholen:1994}
\textsc{G. Janelidze and W. Tholen}, Facets of descent, I, \textit{ Appl. Categ. Structures} \textbf{2} (1994), 245--281.

\bibitem{janelidze.tholen:2007}
\textsc{G. Janelidze and W. Tholen}, Characterization of torsion theories in general categories, \textit{Contemp. Math.} \textbf{431} (2007), 249--256.

\bibitem{kopylov.wegner:2012}
\textsc{Y. Kopylov and S.-A. Wegner}, On the notion of a semi-abelian category in the sense of Palamodov, \textit{Appl. Categ. Structures} \textbf{20} (2012), 531--541.

\bibitem{lue:1967}
\textsc{A.S.-T. Lue}, Baer-invariants and extensions relative to a variety, \textit{Proc. Cambridge Philos. Soc.} \textbf{63} (1967), 569--578.

\bibitem{maclane:1971}
\textsc{S. Mac Lane}, Categories for the working mathematician, \textit{Graduate Texts in Math.} \textbf{6}, Springer (1971). 

\bibitem{palamodov:1971}
\textsc{V.P. Palamodov}, Homological methods in the theory of locally convex spaces, \textit{Uspekhi Mat. Nauk. } \textbf{26} (1971), 3--65 (Russian). 

\bibitem{raikov:1969}
\textsc{D.A. Ra\v{i}kov}, Semiabelian categories, \textit{Soviet Math. Doklady} \textbf{10} (1969), 1242--1245.

\bibitem{rosicky.tholen:2007}
\textsc{J. Rosick\'y and W. Tholen}, Factorization, fibraton and torsion, \textit{J. Homotopy Relat. Struct.} \textbf{2} 2 (2007), 295--314.

\bibitem{rump:2001}
\textsc{W. Rump}, Almost abelian categories, \textit{Cah. Topol. G\'eom. Diff\'er. Cat\'eg.} \textbf{42} 3 (2001), 163--225.

\bibitem{rump:2008}
\textsc{W. Rump}, A counterexample to Raikov's conjecture, \textit{Bull. London Math. Soc.} \textbf{40} 6 (2008), 985--994.

\bibitem{schneiders:1999}
\textsc{J.-P. Schneiders}, Quasi-abelian categories and sheaves, \textit{M\'em. Soc. Math. Fr. Nouv. S\'{e}r.} \textbf{76} (1999).

\bibitem{yoneda:1960}
\textsc{N. Yoneda}, On Ext and exact sequences, \textit{J. Fac. Sci. Univ. Tokyo Sect.} \textbf{18} (1960), 507--576. 

\end{thebibliography}
\end{document}